\documentclass[11pt,leqno]{amsproc}
\usepackage[margin=1in]{geometry}
\usepackage{amsmath, amsthm, amssymb}
\usepackage[colorlinks=true, pdfstartview=FitV, linkcolor=blue,citecolor=blue, urlcolor=blue]{hyperref}
\usepackage[abbrev,lite,nobysame]{amsrefs}
\usepackage{times}
\usepackage[usenames,dvipsnames]{color}
\usepackage[T1]{fontenc}
\usepackage[utf8]{inputenc} 
\usepackage{todonotes}
\usepackage{soul}
\usepackage[normalem]{ulem}

\usepackage{scalerel,stackengine}
\stackMath
\newcommand\reallywidehat[1]{%
\savestack{\tmpbox}{\stretchto{%
  \scaleto{%
    \scalerel*[\widthof{\ensuremath{#1}}]{\kern-.6pt\bigwedge\kern-.6pt}%
    {\rule[-\textheight/2]{1ex}{\textheight}}
  }{\textheight}%
}{0.5ex}}%
\stackon[1pt]{#1}{\tmpbox}%
}
\definecolor{labelkey}{rgb}{0,0,1}
\usepackage{dsfont}
\newenvironment{altproof}[1]
{\noindent
	{\em Proof of {#1}}.}
{\nopagebreak\mbox{}\hfill $\Box$\par\addvspace{0.5cm}}

\def\R {\mathbb{R}}

\newcommand{\eps}{\varepsilon}

\def\div{\mathop{\mathrm{div}\,}}

\def\supp{\mathrm{supp}}

\def\Z{\mathbb{Z}}
\def\d{\diamond}

\def\N{\mathbb{N}}
\def\R{\mathbb{R}}

\def\D{\mathcal{D}}

\newcommand{\cD}{{\mathcal D}}

\newcommand{\cM}{{\mathcal M}}

\newcommand{\cP}{{\mathcal P}}

\newcommand{\Om}{\Omega}

\newcommand{\weakto}{\rightharpoonup}

\newcommand{\tu}{\widetilde{u}}

\def\pr{\right )}
\def\le{\left (}
\def\gg{^{\ast\ast}}
\def\d{\,d}
\def\D{\mathcal{D}^{m,2}(\R^N)}
\def\f{\varphi}
\def\supp{\mathrm{supp}\,}
\def\e{\varepsilon}

\newtheorem{proposition}{Proposition}[section]
\newtheorem{theorem}[proposition]{Theorem}

\newtheorem{lemma}[proposition]{Lemma}
\theoremstyle{definition}

\numberwithin{equation}{section}

\title[Polyharmonic Nonlinear Scalar Field Equations]{Polyharmonic Nonlinear Scalar Field Equations}

\author[A. Cannone]{Alessandro Cannone}
\address[A. Cannone]{\newline\indent
Dipartimento di Matematica,
\newline\indent
Università degli Studi di Bari Aldo Moro,
\newline\indent
Via Orabona 4, 70125, Bari, Italy
\newline\indent
}
\email{\href{mailto: alessandro.cannone@uniba.it}{alessandro.cannone@uniba.it}}

\author[S. Cingolani]{Silvia Cingolani}
\address[S. Cingolani]{\newline\indent
Dipartimento di Matematica,
\newline\indent
Università degli Studi di Bari Aldo Moro,
\newline\indent
Via Orabona 4, 70125, Bari, Italy
\newline\indent
}
\email{\href{mailto: silvia.cingolani@uniba.it}{silvia.cingolani@uniba.it}}

\author[J. Mederski]{Jaros\l aw Mederski}
\address[J. Mederski]{\newline\indent
Institute of Mathematics,
	\newline\indent 
	Polish Academy of Sciences,
	\newline\indent 
	ul. \'Sniadeckich 8, 00-656 Warsaw, Poland
	\newline\indent  
	and
	\newline\indent 
	Faculty of Mathematics and Computer Science,		\newline\indent 
	Nicolas Copernicus University, \newline\indent ul. Chopina 12/18,		 87-100 Toruñ, Poland
}
\email{\href{mailto:jmederski@impan.pl}{jmederski@impan.pl}}



\begin{document}

\begin{abstract}
In this paper, we present a result on the existence of ground state solutions for the polyharmonic nonlinear equation $(-\Delta)^m u=g(u)$,
assuming that  $g$ has a general subcritical growth at infinity, inspired by  Berestycki and Lions \cite{BerestyckiLions}. In comparison with the biharmonic case studied in \cite{Med-Siem}, the presence of a higher-order operator gives rise to several analytical challenges, which are overcome in the present work.
Furthermore, we establish a new polyharmonic logarithmic Sobolev inequality.
\end{abstract}

\maketitle

\section{Introduction}  In this paper  we are interested in  the existence of solution of the following {\em polyharmonic nonlinear equation}
\begin{equation}\label{problem}
(-\Delta )^m u = g(u)\quad \textit{in}\quad  \mathbb{R}^N,
\end{equation}where $m \in \N$, $m\geq 2$,  $N>2m$ and $(-\Delta)^m$ is the polyharmonic operator defined below. Here $g :  \R \to \R$ satisfies the following condition \begin{itemize}
	\item[$(g0)$] $g$ is continuous and there is a constant $c>0$ such that $|g(s)|\leq c\big(1+|s|^{2^{*}-1}\big)\quad \quad$ for        $ s\in\R$,
\end{itemize}where $2^*:=\frac{2N}{N-2m}$.\\ 
\indent
The study of higher-order differential elliptic operators arise in a variety of physical contexts. A classical example is the biharmonic equation (with $m=2$) describing the bending of thin elastic plates in the Kirchhoff–Love theory \cite{love1927elasticity}. In more advanced models, such as strain-gradient elasticity and media with microstructure, higher-order terms are required to describe internal stresses and deformations~\cite{mindlin1964microstructure,Antman}. Other physical context are as follows:  low Reynolds number hydrodynamics, in
structural engineering \cite{Selvadurai,Meleshko} and nonlinear optics \cite{Fibich}. We also mention that high order operators arise to model phase separation of binary fluids. We mention \cite{CH}, where Cahn-Hilliard equation was introduced. Recently applications of this model include complex fluids and soft matter, such as are found in polymer science (see \cite{MMR,R}). A mathematical results concerning nonlinear problems of the form \eqref{problem} are presented in \cite{Gazzola}, see also the references therein. Also in \cite{Grunau} it has been obtained a  Boggio's formula to “polyharmonic” operators of any fractional order $s > 0 $.\\
\indent
It is important to note that techniques applicable to second-order problem -- such as those involving the Laplacian operator $-\Delta$ -- may not extend naturally to higher-order settings. For example, it is well established that the biharmonic operator $(-\Delta)^2 = \Delta^2$ falls outside the scope of certain classical tools, including maximum principles and Polya–Szegő-type inequalities. Moreover, even if $(\Delta u)^2$ belongs to $L^1(\mathbb{R}^N)$, this does not guarantee that $\Delta |u|$ is locally integrable, i.e., $\Delta |u| \notin L^1_{\text{loc}}(\mathbb{R}^N)$ may still occur. For higher order problems, i.e. with $m>2$, we encounter much more difficulties, e.g. the Green’s function may change sign, even in simple domains, complicating representation formulas and positivity arguments, we have weaker regularity properties, Sobolev inequalities and embeddings become less favorable. Although the problem \eqref{problem} was recently studied in the biharmonic case in \cite{Med-Siem}, the higher-order case remained open due to difficulties that have been overcome in the present work. Moreover, we also obtaine a new polyharmonic logarithmic Sobolev inequality, which generalizes the classical logarithmic Sobolev inequality given in \cite{Weissler}.

\indent
Using a variational approach, we will look for solutions to our problem in the Sobolev space $\D$, defined as the closure of space $C_0^{\infty}(\R^N)$ compared to the norm $
\|\cdot\|_{D^{m,2}} := 
\left(
\sum_{|\alpha|=m} \|D^{\alpha} \cdot\|_{L^2(\mathbb{R}^N)}^2
\right)^{\frac{1}{2}}
$.\\ \noindent

We obtain  a necessary condition on the solutions of the polyharmonic problem, that is the following general Pohožaev-type result, in particular see \cite{Med-Siem} for the case $m=2$. We would like to point out that, unlike the case of the biharmonic, with the polyharmonic we are dealing with a high number of derivatives and we must be careful in using the correct convergence estimates, which will be shown below.
\begin{theorem}\label{th:Poho}
    Let $u \in \D$ be a weak solution to \eqref{problem}, where $g$ satisfies $(g_0)$. Then
\[
\int_{\mathbb{R}^N} |\nabla^m u|^2 \, dx = \frac{2N}{N - 2m} \int_{\mathbb{R}^N} G(u)  \, dx, \tag{1.3}
\]
provided that $G(u) \in L^1(\mathbb{R}^N)$, where $G(s) := \int_0^s g(t) \, dt$, $x \in \mathbb{R}^N$, $t \in \mathbb{R}$  and  
\[
\begin{aligned}	
	\nabla^m u:=
	\begin{cases}
		\Delta^{\frac{m}{2}} u  &\text{if m is even,}\\
		\nabla \Delta^{\frac{m-1}{2}} u &\text{if m is odd.}
	\end{cases}
\end{aligned}
\]
\end{theorem}
Next, we show the existence of weak solutions to \eqref{problem} under growth assumption at $0$ and at infinity inspired by a seminal paper due to Berestycki and Lions \cite{BerestyckiLions}  for the  scalar field equation and see also \cite{Mederski,MederskiNon2020}. We assume that $g(0)=0$ and $(g0)$ holds.
Let 
\[
\begin{aligned}	
	G_+(s):=
	\begin{cases}
		\int_0^s \max\{g(t),0\}\, dt  &\text{for } s\geq 0,\\
		\int_{s}^0 \max\{-g(t),0\}\, dt&\text{for }s<0,
	\end{cases}
\end{aligned}
\]
and $g_+(s)=G_+'(s)$. Suppose in addition that
 the following conditions are satisfied:
\begin{itemize}
	\item[$(g1)$] $\lim_{s\to 0}G_+(s)/|s|^{2^{*}}=0$,
	\item[$(g2)$] there exists  $\xi_0>0$ such that $G(\xi_0)>0$,
	\item[$(g3)$] $\lim_{|s|\to \infty}G_+(s)/|s|^{2^{*}}=0$.
\end{itemize}
We introduce  the Poho\v{z}aev manifold
\begin{equation}\label{def:Poh}
\cM :=\Big\{u\in \cD^{2,2}(\R^N)\setminus\{0\}: \int_{\R^N}|\nabla^m u|^2=2^{*}\int_{\R^N}G(u)\, dx\Big\},
\end{equation}
and in view of Theorem \ref{th:Poho}, $\cM$ contains all nontrivial solutions. Let us define the functional $ J(u): \D \to \R$  associated to \eqref{problem} by
\begin{equation}\label{eq:action}
J(u):=\frac12\int_{\R^N} |\nabla^m u|^2- \int_{\R^N} G(u)\, dx.
\end{equation}
Clearly, $J$ is of class $C^1$ and its critical points correspond to weak solutions to \eqref{problem}.

The existence result reads as follows.
\begin{theorem}\label{thm:2}
Let $(g0)$--$(g3)$ be satisfied and $g(0)=0$.
Then $\inf_{\cM} J>0$ and there is a ground state solution $u_0\in \D$ to \eqref{problem}, i.e. $u_0\in\cM$ solves \eqref{problem}
 and $J(u_0) = \inf_{\cM}J$. Moreover $u_0\in C^{2m-1,\alpha}_\text{loc}(\R^N)\cap W^{2m,q}_\text{loc}(\R^N)$, for any $0<\alpha < 1$ and $1\leq q<\infty$.
\end{theorem}
Theorem \ref{thm:2} enables us to consider the following nonlinearity \begin{equation}\label{log_nonlinearity}
G(s)=s^2\log |s|\qquad \textit{for} \, \, \, s\not=0 \, \, \, \textit{and } \, \, \, G(0)=0,
    \end{equation}
which satisfies $(g_0)-(g_3)$ conditions. In view of Theorem \ref{thm:2}, there is a ground state solution to \eqref{problem} and
\begin{equation*}
    C_{N,log}:=2^*\Big( \frac{1}{2}-\frac{1}{2^*}\Big)^{-\frac{2m}{N-2m}} (\inf\limits_{\mathcal{M}}J)^{\frac{2m}{N-2m}}.
\end{equation*}Finally we obtain a new polyharmonic logarithmic Sobolev inequality,  whose optimizers are related to the ground state solutions of the above equation \eqref{problem}. 

\begin{theorem}\label{th1.4}
  For any $u \in \mathcal{D}^{m,2}(\mathbb{R}^N)$ such that $\int_{\mathbb{R}^N} |u|^2 \, dx=1$, there holds \begin{equation}\label{eq: 1.7}
   \frac{N}{4m}\bigg(\bigg( \frac{4m e }{C_{N,\log}(N-2m)}\bigg)^{\frac{N-2m}{N}} \int_{\mathbb{R}^N} |\nabla^m u|^2 \, dx\bigg) \geq \int_{\mathbb{R}^N} |u|^2 \log |u|\, dx 
  \end{equation} and \begin{equation*}
      \bigg( \frac{4m e}{C_{N, \log}(N-2m)}\bigg)^{\frac{N-2m}{N}}< \bigg( \frac{2}{\pi e N}\bigg)^m
  \end{equation*} 
\end{theorem}
Here a crucial observation is an inequality connecting polyharmonic operators -- see Lemma \ref{Poly-inequality} below.

\noindent Moreover, the equality in \eqref{eq: 1.7} holds provided that $u=\frac{u_0}{\| u_0\|_{L^2(\R^N)}}$ and $u_0$ is a ground state solution to \eqref{problem}. If the equality in \eqref{eq: 1.7} holds for $u$, then there are uniquely determined $\lambda >0$ and $r>0$ such that $u_0=\lambda u_0(r \cdot) \in \mathcal{M}$ and $u_0$ is a ground state
solution to \eqref{problem}.\\ \noindent Recall that the classical logarithmic Sobolev inequality given in \cite{Weissler}.
\begin{equation}\label{classical_log_ineq}
\frac{N}{4}\log\Big( \frac{2}{\pi e N} \int_{\R^N}|\nabla u|^2\, dx\Big) \geq \int_{\R^N} u^2\log |u| \, dx\, \, \, \, \, \textit{for} \, \, \, u \in H^1(\R^N), \, \, \int_{\R^N} |u|^2\, dx=1
\end{equation}which is equivalent to the Gross inequality \cite{Gross}.
The sharpness of inequality \eqref{classical_log_ineq} and the identification of its minimizers have previously been established by Carlen \cite{Carlen} in the framework of the classical Gross inequality, as well as by del Pino and Dolbeault \cite{DelPino,DelPinoJMPA} in the context of interpolated Gagliardo–Nirenberg inequalities and the logarithmic Sobolev inequality in $L^p$-spaces. A broader formulation of the optimal Gross inequality in Orlicz spaces was later provided by Adams \cite{Adams}. In the biharmonic case we refer to \cite{Med-Siem}.
To the best of our knowledge, however, logarithmic Sobolev inequalities involving higher-order operators have not yet been addressed in the literature. Thus, inequality \eqref{eq: 1.7} appears to be the first of its kind in the setting of the polyharmonic Laplacian. 
One of the crucial steps is to observe that for a function $u\in \cD^{m+1,2}(\R^N)$ normalized in $L^2$ in power $2/m$ is strictly smaller than the norm in power $2/(m+1)$ involving derivatives of order $m+1$; see Lemma \ref{Poly-inequality} below. This highlights the increase of these terms as the order of differentiation increases and enables us to obtain an estimate of the constant involving $C_{N,\log}$ appearing on the left-hand side of \eqref{eq: 1.7}.\\
\indent 
It is worth also noting that, unlike in the case of \eqref{classical_log_ineq}, explicit ground state solutions to \eqref{problem} with \eqref{log_nonlinearity} are not known for $m>1$. Consequently, the precise value of the constant $C_{N,\log}$ remains an open problem.\\
\indent The structure of the paper is as follows. In Section \ref{sec:Poh}, we establish a Poho\v{z}aev-type identity for general nonlinearity satisfying only $(g0)$. Section \ref{sec:Lions} presents a general version of Lions’ lemma in the space 
$\cD^{m,2}(\R^N)$ (Lemma \ref{lem:Lions}), which plays a key role in the proof of Theorem \ref{thm:2} in Section \ref{sec:BL}. Finally, Section \ref{sec:Log} is dedicated to the proof of the polyharmonic logarithmic Sobolev inequality.

\section{Pohozaev Identity and regularity }\label{sec:Poh}
\noindent  In this section, we prove the Pohozaev identity for \eqref{problem} (cf. \cite{BMS, Med-Siem, PucciSerrin}).  
Before stating the result, we first introduce some notation.  
We begin by defining the polyharmonic operator as follows: $$(-\Delta)^m u:=(-1)^m\sum_{|\alpha|=2m}\frac{m!}{\alpha_1!\ldots\alpha_N!}  \frac{\partial^{|\alpha|}}{\partial x_1^{2\alpha_1}\ldots \partial x_N^{2\alpha_N}} u,\;\alpha=(\alpha_1,\ldots, \alpha_N) \in \R^N,\,\,\, |\alpha|=|\alpha_1|+\cdots +|\alpha_N|.$$ \\ \noindent Let $N, k \in \mathbb{N}$ and $1 \leq p < \infty$ with $N > kp$. We define the spaces $W^{k,p}(\R^N)$ and  ${\mathcal D}^{k,p}(\mathbb{R}^N)$ as the completion of the space $C_0^\infty(\mathbb{R}^N)$ with respect to the norm $\|\cdot \|_{W^{k,p}}$ and $\|\cdot \|_{D^{k,p}}$ respectively, defined as follows
\[
\|u \|_{W^{k,p}}=\Big(\sum_{|\alpha|\leq k} \| D^{\alpha}u\|_{L^p(\R^N)}^p\Big)^{\frac{1}{p}}\, , \, \, u\in C_0^{\infty}(\R^N),
\]
\[
\|u\|_{{\mathcal D}^{k,p}} := 
\left(
\sum_{|\alpha|=k} \|D^{\alpha} u\|_{L^p(\mathbb{R}^N)}^p
\right)^{\frac{1}{p}}, \quad u \in C_0^\infty(\mathbb{R}^N),
\]
 \\ \noindent Clearly, $\|\cdot\|_{{\mathcal D}^{k,p}}$ is a seminorm. The fact that $\|\cdot\|_{{\mathcal D}^{k,p}}$ is indeed a norm follows from the Gagliardo-Nirenberg-Sobolev inequalities
\[
\|u\|_{L^{\frac{Np}{N-kp}}(\mathbb{R}^N)} \leq c \|u\|_{{\mathcal D}^{k,p}}, \quad u \in C_0^\infty(\mathbb{R}^N),
\]
and
\[
\|u\|_{{\mathcal D}^{k-l, \frac{Np}{N-lp}}} \leq c \|u\|_{{\mathcal D}^{k,p}}, \quad u \in C_0^\infty(\mathbb{R}^N), \, 0 \leq l \leq k.
\]
Hence:
\begin{equation} \label{eq:0.1}
{\mathcal D}^{k,p}(\mathbb{R}^N) \subseteq {\mathcal D}^{k-l, \frac{Np}{N-lp}}(\mathbb{R}^N), \quad 0 \leq l \leq k,
\end{equation}
and
\begin{equation} \label{eq:0.2}
\sum_{j=0}^k \sum_{|\alpha|=k-j} \|\partial^\alpha u\|_{L^{\frac{Np}{N-jp}}(\mathbb{R}^N)} \leq c \|u\|_{{\mathcal D}^{k,p}}, \quad u \in {\mathcal D}^{k,p}(\mathbb{R}^N).
\end{equation}
The space $\D$ is of particular importance in our case. Let $\hat{u}(\xi)$ with $\xi=(\xi_1,\ldots,\xi_N)$ stands for the Fourier transform of $u$. In view of the Plancherel theorem, we have:
\begin{equation*}
\begin{aligned}
 \left\|\frac{\partial^m u}{\partial x_{i_1} \partial x_{i_2} \cdots\partial x_{i_m}}\right\|_{L^2(\mathbb{R}^N)} 
& = 
\left\| i \xi_{i_1} i \xi_{i_2}\cdots i \xi_{i_m} \hat{u} \right\|_{L^2(\mathbb{R}^N)} 
 = \int_{\R^N} | \xi_{i_1}  \xi_{i_2}\cdots  \xi_{i_m} \hat{u}|^2\, d\xi\\& \leq \int_{\R^N}\Bigr( \frac{\xi^2_{i_1}+\xi^2_{i_2}+\cdots + \xi^2_{i_m}}{m}\Big)^m |\hat{u}|^2\, d\xi
 \\&=\frac{1}{m^{\frac{m}{2}}}\||\xi|^m \hat{u} \|_{L^2(\R^N)} =\frac{1}{m^{\frac{m}{2}}}\| \reallywidehat{\nabla^m u} \|_{L^2(\R^N)} \\ &=\frac{1}{m^{\frac{m}{2}}}\| \nabla^m u \|_{L^2(\R^N)}, 
\end{aligned}
\end{equation*}

\noindent for all $1\leq i_1<i_2<\ldots i_m \leq N$ and $u \in C_c^\infty(\mathbb{R}^N)$. Hence, there exists a constant $c > 0$ such that, for  $u \in C_c^\infty(\mathbb{R}^N)$,
\[
\frac{1}{c} \|u\|_{\D} \leq \|\nabla^m u\|_{L^2(\mathbb{R}^N)} \leq c \|u\|_{\D},
\]
Therefore, the norms $\|u\| := \|\nabla^m u\|_{L^2(\mathbb{R}^N)}$ and $\|u\|_{\D}$ are equivalent on ${\mathcal D}^{m,2}(\mathbb{R}^N)$.
Moreover, ${\mathcal D}^{m,2}(\mathbb{R}^N)$ is a Hilbert space with the inner product:
\[
\langle u, v \rangle := \int_{\mathbb{R}^N} \nabla^m u \nabla^m v \, dx, \quad \text{for } u, v \in D^{m,2}(\mathbb{R}^N).
\]

To prove the Pohozaev identity, it is useful to recall an important Brezis–Kato type regularity result for solutions of higher-order nonlinear elliptic equations, obtained in \cite{Siemi}; cf. \cite{BrezisKato}.  
\begin{theorem}[\cite{Siemi}]\label{theo_regularity}
	Let $u \in W^{m,2}_{\text{loc}}(\Omega)$ be a weak solution of
	\[
	(-\Delta)^m u = g(x, u(x)) \quad \text{in } \Omega.
	\]
	
	If $g$ satisfies \begin{equation} \label{growth_condition}
	|g(x, u(x))| \leq a(x) \left( 1 + |u| \right), \quad \text{for }  \text{ a.e. } x \in \Omega,
	\end{equation}
	where $0 \leq a \in L^{N/2m}_{\text{loc}}(\Omega), $ then $u \in \bigcap_{1 \leq q < \infty} L^q_{\text{loc}}(\Omega)$. If $g$ has a polynomial rate of growth:
	\[
	|g(x, s)| \lesssim 1 + |s|^{2^*-1}, \quad \text{for all } x \in \Omega \text{ and } s \in \mathbb{R},
	\]
	where $2^*=\frac{2N}{N - 2m}$. Then
	\[
	u \in W^{2m,q}_{\text{loc}}(\Omega) \cap C^{2m-1,h}_{\text{loc}}(\Omega), \quad \text{for all } 1 \leq q < \infty \text{ and } 0 < h < 1.
	\]
\end{theorem}

Using the above result, we impose only $(g0)$ on $g$, which allows us to consider a more general nonlinearity than those studied in \cite{PucciSerrin, Gazzola}.

\noindent \begin{theorem}
Let $u \in \D$ be a weak solution to  \eqref{problem} where $g$ satisfies $(g_0)$.
Then
\begin{equation}\label{pohozaev identity}
\int_{\mathbb{R}^N} (N - 2m) |\nabla^m u|^2  - 2N G(u) \, dx = 0, 
\end{equation} provided $G(u) \in L^1(\R^N)$ where $G(u) := \int_0^u g(t) \, dt$.
\end{theorem}
\begin{proof}
    Let $u \in  D^{m,2}(\mathbb{R}^N)$ weak solution to \eqref{problem}, we can take take $a(x):=g(u(x))/(1+|u(x)|)\in L^{\frac{N}{2m}}_{loc}(\R^N)$ for $u\in L^{2^*}_{loc}(\R^N)$ and by Theorem \ref{theo_regularity}we have   $u$ belongs to
$C^{2m-1,\alpha}_\text{loc}(\R^N)\cap W^{2m,q}_\text{loc}(\R^N)$.  \\  For every $n \geq 1$ let $\psi_n \in C^{\infty}_0(\mathbb{R}^N)$ be radially symmetric such that $0 \leq \psi_n \leq 1$, $\psi_n(x) = 1$ for every $|x| \leq n$, $\psi_n(x) = 0$ for every $|x| \geq 2n$, and $|x||\nabla\psi_n(x)| \leq 1$ for every $x \in \mathbb{R}^N$. Observing that, for every $i \in \mathbb{N}$ 
\begin{equation}\label{i-laplacian}
    \Delta^i (\nabla u \cdot x) = 2i \Delta^i u + \nabla \Delta^i u \cdot x
\end{equation}
and 
\begin{equation*}
\begin{aligned}
{g(u)}(\nabla u \cdot x)\psi_n &= \text{div}\left(\psi_n G(u)x\right) - N\psi_n G(u) - G(u) \nabla \psi_n \cdot x.
\end{aligned}
\end{equation*}
Arguing as in \cite{BMS}, if $m$ is odd or even respectively, one obtains
\begin{equation*}
\begin{aligned}
\Delta^{2k+1} u (\nabla u \cdot x) \psi_n &= \textit{div}\Big(\Big(\Delta^k (x \cdot \nabla u) \nabla \Delta^k u - \frac{|\nabla \Delta^k u|^2}{2}x  - \sum_{j=0}^{k-1} \Delta^{2k-j} u \nabla \Delta^j (\nabla u \cdot x)\\
&\quad + \sum_{j=0}^{k-1} \Delta^j (\nabla u \cdot x) \nabla \Delta^{2k-j} u \Big) \psi_n\Big)  + \frac{N - 4k - 2}{2} |\nabla \Delta^k u|^2 \psi_n -  \Big(\Delta^k (\nabla u \cdot x) \nabla \Delta^k u\\
&\quad - \frac{|\nabla \Delta^k u|^2}{2}x    - \sum_{j=0}^{k-1} \Delta^{2k-j} u \nabla \Delta^j (\nabla u \cdot x) + \sum_{j=0}^{k-1} \Delta^j (\nabla u \cdot x) \nabla \Delta^{2k-j} u \Big) \cdot \nabla \psi_n.
\end{aligned}
\end{equation*}
\begin{equation*}
\begin{aligned}
\Delta^{2k} u (\nabla u \cdot x) \psi_n &= \text{div}\Big(\Big(\frac{1}{2} (\Delta^k u)^2 x + (\nabla u \cdot x) \nabla \Delta^{2k-1} u  + \sum_{j=0}^{k-2} \Delta^{j+1} (\nabla u \cdot x) \nabla \Delta^{2k-j-2} u \\
&\quad  - \sum_{j=0}^{k-1} \Delta^{2k-j-1} u \nabla \Delta^j (\nabla u \cdot x) \Big) \psi_n\Big) + \frac{4k - N}{2} (\Delta^k u)^2 \psi_n \\
&\quad - \Big( \frac{1}{2} (\Delta^k u)^2 x + (\nabla u \cdot x) \nabla \Delta^{2k-1} u  + \sum_{j=0}^{k-2} \Delta^{j+1} (\nabla u \cdot x) \nabla \Delta^{2k-j-2} u \\
&\quad - \sum_{j=0}^{k-1} \Delta^{2k-j-1} u \nabla \Delta^j (\nabla u \cdot x) \Big) \cdot \nabla \psi_n.
\end{aligned}
\end{equation*}
Multiplying both sides of $(-\Delta)^m u=g(u)$ by $\psi_n (\nabla u \cdot x)$ and using the identities above, we have for $m=2k$ 
\begin{equation*}
    \begin{aligned}
        &0=(-(-\Delta)^{2k} u +g(u)) \psi_n (\nabla u \cdot x)\\ &   0=-\div \Big(\Big(\frac{1}{2} (\Delta^k u)^2 x + (\nabla u \cdot x) \nabla \Delta^{2k-1} u + \sum_{j=0}^{k-2} \Delta^{j+1} (\nabla u \cdot x) \nabla \Delta^{2k-j-2} u \\
& - \sum_{j=0}^{k-1} \Delta^{2k-j-1} u \nabla \Delta^j (\nabla u \cdot x) \Big) \psi_n \Big) -\frac{4k - N}{2} (\Delta^k u)^2 \psi_n  + \Big( \frac{1}{2} (\Delta^k u)^2 x + (\nabla u \cdot x) \nabla \Delta^{2k-1} u  \\
& + \sum_{j=0}^{k-2} \Delta^{j+1} (\nabla u \cdot x) \nabla \Delta^{2k-j-2} u  - \sum_{j=0}^{k-1} \Delta^{2k-j-1} u \nabla \Delta^j (\nabla u \cdot x) \Big) \cdot \nabla \psi_n  \\ &\quad +\text{div}\Big(\psi_n G(u)x\Big) - N\psi_n G(u) - G(u) \nabla \psi_n \cdot x.
    \end{aligned}
\end{equation*}
Fix $n\geq 1$ and $R>0$ such that $\supp \psi_n \subset B_R$, we have by divergence Theorem 
\begin{equation*}
    \begin{aligned}
        & 0=\int_{B_R}\frac{ N-4k}{2} (\Delta^k u)^2 \psi_n\, dx  + \int_{ B_{R}} \Big( \frac{1}{2} (\Delta^k u)^2 x +  (\nabla u \cdot x) \nabla \Delta^{2k-1} u  \\
&\quad + \sum_{j=0}^{k-2} \Delta^{j+1} (\nabla u \cdot x) \nabla \Delta^{2k-j-2} u  - \sum_{j=0}^{k-1} \Delta^{2k-j-1} u \nabla \Delta^j (\nabla u \cdot x) \Big) \cdot \nabla \psi_n\, dx  \\ &\quad -\int_{B_R} N\psi_n G(u)\, dx  -\int_{B_R} G(u) \nabla \psi_n \cdot x\, dx
    \end{aligned}
\end{equation*}
We note that 
\begin{equation}
    \begin{aligned}\label{before integration by part}
      \int_{B_R} (\nabla u \cdot x)\nabla \Delta^{2k-1}u \cdot \nabla \psi_n\, dx=& -\int_{B_R} \Delta^{2k-1} u \nabla \psi_n \cdot \nabla(\nabla u \cdot x)\, dx- \int_{B_R} \Delta^{2k-1} u \Delta \psi_n \nabla u \cdot x\, dx\\ &+\int_{B_R} \div ( x\cdot \nabla u  \Delta^{2k-1} u \nabla \psi_n)\, dx,
    \end{aligned}
\end{equation} by a divergence Theorem we have the last  term is equal to 0.\\ \indent Now, integrating by parts $2(k-1)$-times the remaining two terms of \eqref{before integration by part} we have
\begin{equation}\begin{aligned} \label{integration by parts}
    &\int_{B_R} \Delta^{2k-1} u \Delta \psi_n \nabla u \cdot x\, dx=-\int_{B_R} \nabla  \Delta^{2k} u \cdot \nabla (\Delta \psi_n \nabla u \cdot x)\, dx\\ & = (-1)^2\int_{B_R} \Delta^{2k} u \Delta( \Delta \psi_n \nabla u \cdot x)\, dx= \ldots=(-1)^{2(k-1)} \int_{B_R} \Delta^k u \Delta^{k-1}(\Delta \psi_n x \cdot \nabla u)dx.\\ &\int_{B_R} \Delta^{2k-1} u \nabla \psi_n \cdot \nabla(\nabla u \cdot x)\, dx= -\int_{B_R} \nabla \Delta^{2k} u \cdot \nabla (\nabla \psi_n \cdot \nabla(\nabla u \cdot x))\, dx \\ & =(-1)^2\int_{B_R} \Delta^{2k} u \Delta(\nabla \psi_n \cdot \nabla(\nabla u \cdot x))\, dx=\ldots= (-1)^{2(k-1)}\int_{B_R} \Delta^{k} u \Delta^{k-1} (\nabla \psi_n \cdot \nabla(\nabla u \cdot x))\, dx. 
\end{aligned}
\end{equation}
\\ \noindent Extending the integrals over the whole space, we have \begin{equation} \label{pohozaev_psi_n_pari}
\begin{aligned}
     0=&\int_{\R^N}\frac{ N-4k}{2} (\Delta^k u)^2 \psi_n\, dx  + \int_{ \R^N} \bigg( \frac{1}{2} (\Delta^k u)^2 x +   
  \sum_{j=0}^{k-2} \Delta^{j+1} (\nabla u \cdot x) \nabla \Delta^{2k-j-2} u\\ &-\sum_{j=0}^{k-1} \Delta^{2k-j-1} u \nabla \Delta^j (\nabla u \cdot x) \bigg) \cdot \nabla \psi_n\, dx -\int_{\R^N} \Delta^{k} u \Delta^{k-1} (\nabla \psi_n \cdot \nabla(\nabla u \cdot x))\, dx \\ &  - \int_{\R^N} \Delta^k u \Delta^{k-1}(\Delta \psi_n x \cdot \nabla u)dx -\int_{\R^N} N\psi_n G(u)\, dx  -\int_{\R^N} G(u) \nabla \psi_n \cdot x\, dx.
\end{aligned}
   \end{equation}
   To better understand what will be done we need to recall these approximations \begin{equation}\label{k-laplacian scalar}
   \Delta^{k-1}(\Delta \psi_n x \cdot \nabla u) \approx \Delta^{k-1}\Delta\psi_n x \cdot \nabla u+ B_{k-1}(\psi_n, u)+ \Delta \psi_n \Delta^{k-1}(x\cdot\nabla u)
\end{equation}

 \begin{equation}\label{vector identy}\begin{aligned}
     \nabla \psi_n  \cdot \nabla (\nabla u \cdot x)&= \nabla \psi_n \cdot \nabla\Big( \sum_{i=1}^N \frac{\partial}{\partial x_i}u x_i\Big)=\Big(\sum_{i=1}^N \frac{\partial}{\partial x_1}\Big(\frac{\partial}{\partial x_i}u x_i\Big), \ldots, \sum_{i=1}^N \frac{\partial}{\partial x_1}\Big(\frac{\partial}{\partial x_i}u x_i\Big)       \Big)\\ &=\sum_{j=1}^N\sum_{i=1}^N \frac{\partial^2}{\partial x_j\partial x_i}u x_i \frac{\partial}{\partial x_j}\psi_n + \frac{\partial}{\partial x_j}u\frac{\partial }{\partial x_j}\psi_n=\sum_{i,j=1}^N\frac{\partial^2}{\partial x_j\partial x_i}u x_i \frac{\partial}{\partial x_j}\psi_n + \nabla u \cdot \nabla \psi_n. \end{aligned}
\end{equation}
\begin{equation}\label{lps}
    \begin{aligned}
        \Delta^{k-1} (\frac{\partial^2}{\partial x_j\partial x_i}u x_i \frac{\partial}{\partial x_j}\psi_n) & \approx \Delta^{k-1}(\frac{\partial^2}{\partial x_j\partial x_i}u) x_i \frac{\partial}{\partial x_j}\psi_n+ C_{k-1}(u, \psi_n)+ \Delta^{k-1}(x_i \frac{\partial}{\partial x_j}\psi_n)\frac{\partial^2}{\partial x_j\partial x_i}u\\ & \approx \Delta^{k-1}(\frac{\partial^2}{\partial x_j\partial x_i}u) x_i \frac{\partial}{\partial x_j}\psi_n+ C'_{k-1}(u, \psi_n)+ \psi_n^{(2k-1)}\Big(\frac{2 x_i}{n^2}\Big)^{2k-2}\frac{2x_j}{n^2}\frac{\partial^2}{\partial x_j\partial x_i}u.
    \end{aligned}
\end{equation}
In \eqref{k-laplacian scalar} $B_{k-1}$ encompasses all the products of the mixed derivatives between $\psi_n$ and $u$  which are certainly of a lower degree than $\Delta^{k-1} \psi_n$ and  $\Delta^{k-1} u$ respectively, same consideration for $C$ and $C'$ in \eqref{lps}.

By \eqref{i-laplacian}, \eqref{k-laplacian scalar},  \eqref{vector identy} and \eqref{lps},   the first and the second terms becomes of \eqref{integration by parts} become respectively 
\begin{equation}\label{UNO}\begin{aligned}
     \int_{\R^N} \Delta^k u \Delta^{k-1}(\Delta \psi_n x \cdot \nabla u)\, dx&= \int_{\R^N} \Delta^k u \Delta^{k-1}\Delta\psi_n x \cdot \nabla u  + \int_{\R^N} \Delta^k u B_{k-1}(\psi_n, u)\, dx \\ & +\int_{\R^N} \Delta^k u  \Delta  \psi_n( 2(k-1) \Delta^{k-1} u + \nabla \Delta^{k-1} u \cdot x)\, dx.
     \end{aligned}
     \end{equation}
 \begin{equation}\label{DUE}\begin{aligned}
          &\int_{\R^N} \Delta^{k} u \Delta^{k-1} (\nabla \psi_n \cdot \nabla(\nabla u \cdot x))\, dx\\& = \int_{\R^N} \Delta^k u \Delta^{k-1}\Big(  \sum_{i,j=1}^N\frac{\partial^2}{\partial x_j\partial x_i}u x_i \frac{\partial}{\partial x_j}\psi_n + \frac{\partial u}{\partial x_j}\frac{\partial \psi_n}{\partial x_j}\Big)\, dx\\ & =\sum_{i,j=1}^N\int_{\R^N} \Delta^ku \Big(\Delta^{k-1}(\frac{\partial^2}{\partial x_j\partial x_i}u) x_i \frac{\partial}{\partial x_j}\psi_n+ C'_{k-1}(u, \psi_n)+ \psi_n^{(2k-1)}\Big(\frac{2 x_i}{n^2}\Big)^{2k-2}\frac{2x_j}{n^2}\frac{\partial^2}{\partial x_j\partial x_i}u\Big)\, dx\\ &+  \int_{\R^N} \Delta^k u \Delta^{k-1}\frac{\partial u}{\partial x_j}\frac{\partial \psi_n}{\partial x_j} + \Delta^k u B_{k-1}(u,\psi_n) + \Delta^k u \frac{\partial u}{\partial x_j}\Delta^{k-1}\frac{\partial \psi_n}{\partial x_j} \, dx.
         \end{aligned}
       \end{equation}
       Therefore  showing that \eqref{integration by parts} goes to zero is equivalent to showing that both \eqref{UNO} and \eqref{DUE} go to zero. \\
       \indent 
        To facilitate the reader's understanding 
and to lighten the proof, we show the calculations  in some pieces in  \eqref{UNO} and \eqref{DUE}, since the same reasoning can be applied to the remaining terms. \\ Now we show how \eqref{UNO} goes to 0. Before we remark that  
           \begin{equation*}
    \begin{aligned}
&\int_{\R^N\smallsetminus B_n}  |\Delta^{2k-1} u \Delta \psi_n \nabla u \cdot x|\, dx=\int_{\R^N\smallsetminus B_n}|\Delta^k u||\Delta^{k-1}\Delta\psi_n x \cdot \nabla u+ B_{k-1}(\psi_n, u)+ \Delta \psi_n \Delta^{k-1}(x\cdot\nabla u)|\, dx \\ &\leq  \int_{\R^N\smallsetminus B_n}|\Delta^k u||(\Delta^{k-1}\Delta\psi_n x \cdot \nabla u+ B_{k-1}(\psi_n, u))|\, dx\\ &+\int_{\R^N\smallsetminus B_n} |\Delta^k u  \Delta  \psi_n( 2(k-1) \Delta^{k-1} u + \nabla \Delta^{k-1} u \cdot x)|\, dx \\ &\leq  \int_{\R^N\smallsetminus B_n}|\Delta^k u||\Delta^{k-1}\Delta\psi_n x \cdot \nabla u|\, dx+  \int_{\R^N\smallsetminus B_n}|\Delta^k u| B_{k-1}(\psi_n, u)|\, dx+ \int_{\R^N\smallsetminus B_n} c|\Delta^k u|  |\Delta  \psi_n||\Delta^{k-1} u|\, dx  \\ &+ \int_{\R^N\smallsetminus B_n}|\Delta^k u|  |\Delta  \psi_n| |\nabla \Delta^{k-1} u|  |x|\, dx.
\end{aligned}
\end{equation*}
Let's now show how the last term goes to zero. By H\"older inequalities, properties of $\psi_n$ and the embeddings, we have 
\begin{equation}
    \begin{aligned}
          &\int_{\R^N\smallsetminus B_n}|\Delta^k u|  |\Delta  \psi_n| |\nabla \Delta^{k-1} u|  |x|\, dx \\ &\leq \bigg( \int_{\R^N\smallsetminus B_n} |\Delta^k u|^2\, dx \bigg)^{\frac{1}{2}}\bigg(\int_{\R^N \smallsetminus B_n} (|\Delta \psi_n||x|)^{N}\, dx \bigg)^{\frac{1}{N}} \bigg(\int_{\R^N \smallsetminus B_n} |\nabla \Delta^{k-1} u|^{\frac{2N}{N-2}}\, dx \bigg)^{\frac{N-2}{2N}}\\ &
        \leq \bigg( \int_{\R^N\smallsetminus B_n} |\Delta^k u|^2\, dx \bigg)^{\frac{1}{2}} C \to 0, \,\,\,\,\,\,\, \textit{for} \,\,\,\,\,\, n\to +\infty;
    \end{aligned}
\end{equation}
Then, applying the same considerations for the remaining terms of \eqref{UNO} goes to $0$. \\ 

Now we show how \eqref{DUE} goes to zero.  Before we remark that
\begin{equation*}
    \begin{aligned}
&\int_{\R^N\smallsetminus B_n}| \Delta^{k} u \Delta^{k-1} (\nabla \psi_n \cdot \nabla(\nabla u \cdot x))|\, dx\\ &\leq  \sum_{i,j=1}^N \int_{\R^N\smallsetminus B_n} |\Delta^{k} u| \Big(|\Delta^{k-1}(\frac{\partial^2}{\partial x_j\partial x_i}u)| |x_i| | \frac{\partial}{\partial x_j}\psi_n| + |C'_{k-1}(u, \psi_n)|\\ &+ |\psi_n^{(2k-1)}\Big(\frac{2 x_i}{n^2}\Big)^{2k-2}\frac{2x_j}{n^2}\frac{\partial^2}{\partial x_j\partial x_i}u| \Big)\, dx+ \int_{\R^N\smallsetminus B_n}  |\Delta^k u| |\Delta^{k-1}\frac{\partial u}{\partial x_j}\frac{\partial \psi_n}{\partial x_j}| + |\Delta^k u | |B_{k-1}(u,\psi_n)| \\ & + |\Delta^k u| |\frac{\partial u}{\partial x_j}|\Delta^{k-1}\frac{\partial \psi_n}{\partial x_j}| \, dx.
   \end{aligned}
\end{equation*}
We  focus on the terms $ \sum_{i,j=1}^N \int_{\R^N \smallsetminus B_n} |\Delta^k u| |\psi_n^{(2k-1)}\Big(\frac{2 x_i}{n^2}\Big)^{2k-2}\frac{2x_j}{n^2}\frac{\partial^2}{\partial x_j\partial x_i}u|\, dx$.
By H\"older inequalities and for the properties of $\psi_n$ and for the Sobolev embeddings we have for these terms 
\begin{equation*}
    \begin{aligned}
        &  \sum_{i,j=1}^N \int_{\R^N \smallsetminus B_n} |\Delta^k u| |\psi_n^{(2k-1)}\Big(\frac{2 x_i}{n^2}\Big)^{2k-2}\frac{2x_j}{n^2}\frac{\partial^2}{\partial x_j\partial x_i}u|\, dx\\ & = \sum_{i,j=1}^N \int_{n \leq |x| \leq  \sqrt{2}n} \Delta^{k} u \psi_n^{(2k-1)}(\frac{|x|}{n^2}) \Big( \frac{2x_i}{n^2}\Big)^{2k-2}\frac{2x_j}{n^2}\partial^2_{x_jx_i} u  \, dx (\psi_n,u,x) \\& \leq \sum_{i,j=1}^N \int_{n \leq |x| \leq  \sqrt{2}n} \Delta^{k-1} u \psi_n^{(2k-1)}(\frac{|x|}{n^2})  \frac{|x|^{4k}}{n^{4k}}\partial^2_{x_jx_i} u  \, dx + C_{i,j}(\psi_n,u, x)\\
&\leq \sum_{i,j=1}^N 2^{2k} \int_{n \leq |x| \leq  \sqrt{2}n} |\psi_n^{(2k-2)}||\Delta^{k} u||\partial^2_{x_jx_i} u | \, dx \\&  \leq \sum_{i,j=1}^N 2^{2k} \Big(\int_{\R^N\smallsetminus B_n} |\psi_n^{(2k-2)}|^{\frac{N}{m-2}}\, dx\Big)^{\frac{m-2}{N}}\Big( \int_{\R^N\smallsetminus B_n} |\Delta^{k} u|^2\, dx\Big)^{\frac{1}{2}}\Big(\int_{\R^N\smallsetminus B_n}| \partial^2_{x_jx_i} u|^{\frac{2N}{N-2(m-2)}}\, dx \Big)^{\frac{N-2(m-2)}{2N}}\\ &  \leq \sum_{i,j=1}^N 2^{2k} \Big( \int_{\R^N\smallsetminus B_n} |\Delta^{k} u|^2\, dx\Big)^{\frac{1}{2}}\Big(\int_{\R^N\smallsetminus B_n}| \partial^2_{x_jx_i}u|^{\frac{2N}{N-2(m-2)}}\, dx \Big)^{\frac{N-2(m-2)}{2N}} \to 0\,\,\,\,\,\,\,\,\,\,\, \textit{for}\,\,\,\,\,\,\,\,\, n\to +\infty \end{aligned}
\end{equation*}
As for the \eqref{UNO}, applying the same calculations for the remaining pieces of the \eqref{DUE} we have that this goes to zero. \\
\indent 
In order to estimate  $\sum_{j=0}^{k-2} \int_{\R^N} \Delta^{j+1} (\nabla u \cdot x) \nabla \Delta^{2k-j-2} u \cdot \nabla \psi_n\, dx$ and  $\sum_{j=0}^{k-1} \int_{\R^N} ( \Delta^{2k-j-1} u \nabla \Delta^j (\nabla u \cdot x) ) \cdot \nabla \psi_n\, dx$, we can follow the same steps as before. 
 \\ \noindent Summing up for the even case we have by the Lebesgue dominated convergence theorem and for the properties of $\psi_n$, $G$ and $u$\begin{equation*}
    \begin{aligned}
        & \int_{ \R^N} \bigg( \frac{1}{2} (\Delta^k u)^2 x +   
  \sum_{j=0}^{k-2} \Delta^{j+1} (\nabla u \cdot x) \nabla \Delta^{2k-j-2} u  - \sum_{j=0}^{k-1} \Delta^{2k-j-1} u \nabla \Delta^j (\nabla u \cdot x) \bigg) \cdot \nabla \psi_n\, dx\to 0 \,\,\, \textit{for} \,\,\, n\to +\infty\\ & -\int_{\R^N} \Delta^{2k-1} u \nabla \psi_n \cdot \nabla(\nabla u \cdot x)\, dx  - \int_{\R^N} \Delta^{2k-1} u \Delta \psi_n \nabla u \cdot x\, dx\to 0 \,\,\, \textit{for} \,\,\, n\to +\infty \\ & -\int_{\R^N} G(u) \nabla \psi_n \cdot x\, dx \to 0 \,\,\, \textit{for} \,\,\, n\to +\infty \\ & \int_{\R^N}\frac{ N-4k}{2} (\Delta^k u)^2 \psi_n\, dx   -\int_{\R^N} N\psi_n G(u)\, dx \,\,\, \to  \int_{\R^N}\frac{ N-4k}{2} (\Delta^k u)^2\, dx -N\int_{\R^N}  G\, dx \,\,\, \textit{for}\,\,\, n\to +\infty.
    \end{aligned}
\end{equation*}    
 Finally by \eqref{pohozaev_psi_n_pari}, we have \begin{equation}\label{pohozaev_pari}
     \int_{\R^N}\frac{ N-4k}{2} (\Delta^k u)^2\, dx -N\int_{\R^N}  G(u)\, dx=0.
 \end{equation} 
 The same considerations apply  for the all odd case.\\ 
\noindent Finally we have 
\begin{equation*}
0 = \int_{\mathbb{R}^N} \left( -( -\Delta)^m u + g(u) \right) \psi_n \nabla u \cdot x \, dx
\end{equation*}
\begin{equation*}
= \int_{\mathbb{R}^N} ( \frac{1}{2} |\nabla^m u|^2 \nabla \psi_n \cdot x + X \cdot \nabla \psi_n + \frac{N - 2m}{2} \psi_n |\nabla^m u|^2 - N \psi_n G(u)- G(u) \nabla \psi_n \cdot x ) \, dx
\end{equation*}
\begin{equation*}
X := 
\begin{cases}
-\Delta^k (\nabla u \cdot x) \nabla \Delta^k u + \sum_{j=0}^{k-1} \Delta^{2k-j} u \nabla \Delta^j (\nabla u \cdot x) - \sum_{j=0}^{k-1} \Delta^j (\nabla u \cdot x) \nabla \Delta^{2k-j} u, \\
\nabla u \cdot x \nabla \Delta^{2k-1} u + \sum_{j=0}^{k-2} \Delta^{j+1} (\nabla u \cdot x) \nabla \Delta^{2k-j-2} u - \sum_{j=0}^{k-1} \Delta^{2k-j-1} u \nabla \Delta^j (\nabla u \cdot x),
\end{cases}
\end{equation*}
if $m = 2k + 1$ or $m = 2k$ respectively. \\ \noindent   
Finally, from the properties of $\psi_n$ and the dominated convergence theorem, we
conclude the proof by letting $n \to + \infty$.

\end{proof}

\section{Lions lemma}\label{sec:Lions}
We prove the following result, which implies the variant of Lions's lemma in $\D)$, cf. \cite{Lions1,Lions2, Med-Siem}. In fact, we extend the arguments from \cite[Lemma 4.1]{Med-Siem} from the case 
$m=2$ to the general case, and for the reader's convenience, we provide a detailed proof.

\begin{lemma}\label{lem:Conv}
	Suppose that   $(u_n)\subset \D$ is bounded. Then $u_n(\cdot+y_n)\weakto 0$ in $\D$ for any $(y_n)\subset \Z^N$ if and only if
	\begin{equation*}
		\int_{\R^N} \Psi(u_n)\, dx\to 0\quad\hbox{as } n\to\infty
	\end{equation*}
	for any continuous $\Psi:\R\to \R$ satisfying \eqref{eq:Psi}.
\end{lemma}
\begin{proof}
	Let $(u_n)$ be a sequence in $\D$ be such that $u_n(\cdot+y_n)\weakto 0$ in $\D$ for every $(y_n)\subset \Z^N$.
	Take any $\eps>0$ and $2_\ast<p<2^*$, $2_*:=\frac{2N}{N-2(m-1)}$ and suppose that $\Psi$ satisfies \eqref{eq:Psi}. Fix $\eps$, for a definition of limit we find $0<\delta<M$ and $c(\eps)>0$ such that 
	\begin{eqnarray*}
		\Psi(s)&\leq& \eps |s|^{2^*}\quad\hbox{ for }|s|\leq \delta,\\
		\Psi(s)&\leq& \eps |s|^{2^*}\quad\hbox{ for }|s|>M,\\
		\Psi(s)&\leq& c(\eps) |s|^{p}\quad\hbox{ for }|s|\in (\delta,M].
	\end{eqnarray*}
	Let us define $(w_n)$ by
	\[
	w_n(x):=
	\begin{cases}|u_n(x)|^{2^*/2_\ast}\delta ^{1-2^*/2_\ast} &\text{for }|u_n(x)|\leq \delta.,\\
 |u_n(x)| &\text{for }|u_n(x)|>\delta.
	\end{cases}
	\] 
	We want to prove that $(w_n)$ is bounded in $W^{1,2_\ast}(\R^N)$.
	First of all, we have
	\begin{equation}\label{eq:1.30}
		\begin{aligned}
			\int_{\R^N} |w_n(x)|^{2_\ast}\d x &= \int_{\{|u_n|\leq \delta\}}\delta ^{2_\ast - 2^*} |u_n|^{2^*}\d x + \int_{\{|u_n|\geq \delta\}}|u_n|^{2_\ast}\d x\\
			&= \delta ^{2_\ast - 2^*}\int_{\{|u_n|\leq \delta\}}|u_n|^{2^*}\d x + \int_{\{|u_n|> \delta\}}\frac{|u_n|^{2^*}}{|u_n|^{2^*-2_\ast}}\d x\\
			&\leq \delta ^{2_\ast - 2^*}\int_{\{|u_n|\leq \delta\}}|u_n|^{2^*}\d x + \int_{\{|u_n|> \delta\}}\frac{|u_n|^{2^*}}{\delta^{2^*-2_\ast}}\d x \\
			&= \delta^{2^\ast-2^*} \int_{\R^N}|u_n|^{2^*}\d x.
		\end{aligned}	
	\end{equation}
	By the absolute continuous characterization (see \S 1.1.3 in \cite{Mazja}), we infer that each $u_n$ is absolutely continuous on almost every line parallel to the $0x_i$-axis, for $i=1,\ldots, N$. Thus the same holds for each $w_n$, since $w_n = F(u_n)$, where $F(t) = \min\{\delta^{1-2^*/2_\ast}|t|^{2^*/2_\ast},|t|\}$ is a globally Lipschitz function.
	Moreover, for every $i=1,\ldots, N$, we have
	\[
	\frac{\partial w_n}{\partial x_i} = 
	\begin{cases}
		\frac{2^*}{2_\ast}\delta^{1-2^*/2_\ast}\mathrm{sign}(u_n)|u_n|^{2^*/2_\ast-1}\frac{\partial u_n}{\partial x_i}	, &\text{for } |u_n(x)|\leq \delta,\\
		\mathrm{sign}(u_n) \frac{\partial u_n}{\partial x_i},&\text{for }|u_n(x)|> \delta.	
	\end{cases}
	\]
	Thus
	\begin{equation}\label{eq:1.31}
		\begin{aligned}
			\int_{\R^N}\left | \frac{\partial w_n}{\partial x_i}\right|^{2_\ast}\d x&= \int_{\{|u_n|\leq \delta\}} \left| \frac{2^*}{2_\ast}\delta^{1-2^*/2_\ast}\mathrm{sign}(u_n)|u_n|^{2^*/2_\ast-1}\frac{\partial u_n}{\partial x_i} \right|^{2_\ast}\, dx + \int_{\{|u_n|>\delta\}}\left|\frac{\partial u_n}{\partial x_i} \right |^{2_\ast}\d x\\ &= \le \frac{2^*}{2_\ast}\pr ^{2_\ast}\delta^{2_\ast-2^*}\int_{\{|u_n|\leq \delta\}} |u_n|^{2^*- 2_\ast}\left |\frac{\partial u_n}{\partial x_i}\right|^{2_\ast} \d x + \int_{\{|u_n|>\delta\}}\left|\frac{\partial u_n}{\partial x_i} \right |^{2_\ast}\d x\\ & \leq \le \frac{2^*}{2_\ast}\pr ^{2_\ast} \delta^{2_\ast-2^*}\int_{\{|u_n|\leq \delta\}} \delta^{2^*-2_\ast} \left | \frac{\partial u_n}{\partial x_i}\right|^{2_\ast} \, dx + \int_{\{|u_n|>\delta\}}\left|\frac{\partial u_n}{\partial x_i} \right |^{2_\ast}\d x\\ &= \le \frac{2^*}{2_\ast}\pr ^{2_\ast}\int_{\{|u_n|\leq \delta\}} \left |\frac{\partial u_n}{\partial x_i}\right|^{2_\ast} \d x + \int_{\{|u_n|>\delta\}}\left|\frac{\partial u_n}{\partial x_i} \right |^{2_\ast}\d x\\
			&\leq \le \frac{2^*}{2_\ast}\pr ^{2_\ast} \int_{\R^N}\left| \frac{\partial u_n}{\partial x_i}\right|^{2_\ast} \d x.
		\end{aligned}	
	\end{equation}
	By \eqref{eq:1.30}, \eqref{eq:1.31} (again using an absolute continuous characterization on lines from \S 1.1.3 \cite{Mazja}) and the fact that $(u_n)$ is bounded in $\D$, we conclude that $(w_n)$ is bounded in $W^{1,2_\ast}(\R^N)$.
	
	Let  $\Om=(0,1)^N$ and $y\in  \R^N$ be arbitrary.
	Then, by the Sobolev inequality one has 
	\begin{eqnarray*}
		\int_{\Om+y}\Psi(u_n)\d x&=&
		\int_{(\Om+y)\cap\{\delta<|u_n|\leq M\}}\Psi(u_n)\d x
		+\int_{(\Om+y)\cap (\{|u_n|> M\}\cup \{|u_n|\leq \delta\})}\Psi(u_n)\d x\\
		&\leq &
		c(\eps)\int_{(\Om+y)\cap\{\delta<|u_n|\leq M\}}|w_n|^p\d x
		+\eps \int_{(\Om+y)\cap (\{|u_n|> M\}\cup \{|u_n|\leq \delta\})}|u_n|^{2^*}\d x\\
		&\leq &c(\eps) C\Big(\int_{\Om+y}|w_n|^{2_\ast} + |\nabla w_n|^{2_\ast}\d x \Big)\Big(\int_{\Om+y}|w_n|^{p}\d x\Big)^{1-2_\ast/p}+\eps \int_{\Om+y}|u_n|^{2^*}\d x,
	\end{eqnarray*}
	where $C>0$ is a constant from the Sobolev inequality.
	Then we sum the inequalities over $y\in\Z^N$ and get  
	\begin{equation} \label{psi_lions_inequa}
		\int_{\R^N}\Psi(u_n)\d x \leq c(\eps) C\left(\int_{\R^N}|w_n|^{2_\ast} + |\nabla w_n|^{2_\ast}\d x \right)\left(\sup_{y\in\Z^N}\int_{\Om}|w_n(\cdot +y)|^{p}\d x\right)^{1-2_\ast/p}+\eps \int_{\R^N}|u_n|^{2^*}\d x.
	\end{equation} Let us take $(y_n)\subset\Z^N$ such that
	\begin{equation}\label{supy_nlions}  \sup_{y\in\Z^N}\int_{\Om}|w_n(\cdot+y)|^{p}\d x\leq 2\int_{\Om}|w_n(\cdot+y_n)|^{p}\d x 
	\end{equation}
	for any $n\geq 1$.
	We know by assumption that 
	$u_n(\cdot+y_n)\weakto 0$ in $\D$ and therefore exist  a subsequence  that $u_n(\cdot+y_n)\to 0$ in $L^p(\Om)$.
	
	\noindent Since $|w_n(x)|\leq |u_n(x)|$,  we infer that $w_n(\cdot+y_n)\to 0$ in $L^p(\Om)$. 
	Therefore  we calculate the limsup of \eqref{psi_lions_inequa} using the \eqref{supy_nlions} and the limits property we have
\begin{equation*}
    \begin{aligned}
        &\limsup\limits_{n\to +\infty} \int_{\R^N}\Psi(u_n)\d x\\ 	& \leq \limsup\limits_{n\to +\infty} c(\eps) C\left(\int_{\R^N}|w_n|^{2_\ast} + |\nabla w_n|^{2_\ast}\d x \right)\left(2\int_{\Om}|w_n(\cdot+y_n)|^{p}\d x\right)^{1-2_\ast/p} +\limsup\limits_{n\to +\infty}\eps \int_{\R^N}|u_n|^{2^*}\d x\\ & \leq \eps \limsup_{n\to\infty}\int_{\R^N}|u_n|^{2^*}\d x,
    \end{aligned}
\end{equation*}	
	and since $\eps>0$ is arbitrary, the assertion follows.\\
\indent 	
	To prove the second implication, we proceed by contradiction, suppose that  $u_n(\cdot+y_n)$ does not converge to $0$  in $\D$, for some $(y_n)$  in $\Z^N$, and $\Psi(u_n)\to 0$ in $L^1(\R^N)$. We may assume that $u_n(\cdot+y_n)\to u_0\neq 0$ in $L^p(\Om)$ for some bounded domain $\Om\subset\R^N$ and $1<p<2^*.$ Take any $\eps>0$, $q>2^*$ and let us define $\Psi(s):=\min\{|s|^p,\eps^{p-q}|s|^q\}$ for $s\in\R$. Then
	\begin{eqnarray*}
		\int_{\R^N} \Psi(u_n)\d x&\geq& \int_{\Om + y_n\cap \{|u_n|\geq \eps\}}|u_n|^p\d x+\int_{\Om+ y_n\cap \{|u_n|\leq \eps\}} \eps^{q-p}|u_n|^q\d x\\
		&=& \int_{\Om+y_n} |u_n|^p\d x+\int_{\Om+y_n\cap \{|u_n|\leq \eps\}} \eps^{p-q}|u_n|^q-|u_n|^p\d x\\
		&\geq& \int_{\Om+y_n} |u_n|^p\d x - \e^p|\Omega|.\\
	\end{eqnarray*}
	Thus we get $u_n(\cdot + y_n) \to 0$ in $L^p(\Om)$ and this contradicts $u_0\neq 0$.
\end{proof}
\begin{lemma}\label{lem:Lions}
	Suppose that $(u_n)$ is bounded in $\cD^{m,2}(\R^N)$ and for some $r>0$ 	
	\begin{equation}\label{eq:LionsCond11}
		\lim_{n\to\infty}\sup_{y\in \R^N} \int_{B(y,r)} |u_n|^2\,dx=0.
	\end{equation}
	Then  
	$$\int_{\R^N} \Psi(u_n)\, dx\to 0\quad\hbox{as } n\to\infty$$
	for every continuous $\Psi:\R\to \R$ satisfying
	\begin{equation}
		\label{eq:Psi}
		\lim_{s\to 0} \frac{\Psi(s)}{|s|^{2^{*}}}=\lim_{|s|\to\infty} \frac{\Psi(s)}{|s|^{2^{*}}}=0.
	\end{equation}
\end{lemma}

\begin{altproof}
	Suppose that there is $(y_n)\subset \Z^N$ such that $u_n(\cdot+y_n)$ does not converge weakly to $0$ in $\D$. Since  $u_n(\cdot+y_n)$ is bounded, there is $u_0\neq 0$ such that, up to a subsequence,
	$$u_n(\cdot+y_n)\weakto u_0\quad \text{in }\D,$$
	as $n\to\infty$. We find $y\in \R^N$ such that $u_0\chi_{B(y,r)}\neq 0$ in $L^2(B(y,r))$. 
	Observe that, passing to a subsequence, we may assume that $u_n(\cdot+y_n)\to u_0$ in $L^2(B(y,r))$. Then, in view of \eqref{eq:LionsCond11}
	$$\int_{B(y,r)} |u_n(\cdot+y_n)|^2\,dx=\int_{B(y_n+y,r)} |u_n|^2\,dx\to 0$$
	as $n\to\infty$, which contradicts the fact $u_n(\cdot+y_n)\to u_0\neq 0$ in $L^2(B(y,r))$. Therefore $u_n(\cdot+y_n)\weakto 0$ in $\D$ for any $(y_n)\subset \Z^N$ and by Lemma \ref{lem:Conv} we conclude.
\end{altproof}

\section{Proof of Theorem \ref{thm:2}}\label{sec:BL}

In this section, we adapt a variational framework inspired by \cite[Section 3]{Mederski} for the bi-Laplacian operator. Let us define the function
\[
G_-(s) :=
\begin{cases}
\displaystyle\int_0^s \max\{-g(t), 0\}\, dt & \text{if } s \geq 0, \\
\displaystyle\int_s^0 \max\{g(t), 0\}\, dt & \text{if } s < 0.
\end{cases}
\]
Observe that both \( G_+ \) and \( G_- \) are non-negative, and \( G = G_+ - G_- \).\\
\indent We begin by outlining the variational approach via a regularized energy functional \( J_\varepsilon \), along with a few auxiliary lemmas. The proof of Theorem~\ref{thm:2} is deferred to the end of the section.\\
\indent 
Let us define
\[
g_+(s) := G_+'(s), \qquad g_-(s) := g_+(s) - g(s), \qquad s \in \mathbb{R}.
\]
Then,
\[
G_-(s) = \int_0^s g_-(t)\, dt \geq 0, \quad \text{for all } s \in \mathbb{R}.
\]
Assuming conditions (g1) and (g3), there exists a constant \( c > 0 \) such that
\[
|G_+(s)| \leq c |s|^{2^*}, \quad \text{for all } s \in \mathbb{R},
\]
which implies \( G_+(u) \in L^1(\mathbb{R}^N) \) whenever \( u \in \cD^{m,2}(\R^N) \subset L^{2^*}(\mathbb{R}^N) \).
However, in general, \( G_-(u) \) may not be integrable unless it satisfies a similar growth bound, i.e., \( G_-(u) \leq c |u|^{2^*} \). To address this, for \( \varepsilon \in (0,1) \), we introduce a cutoff function \( \varphi_\varepsilon : \mathbb{R} \to [0,1] \) defined by
\[
\varphi_\varepsilon(s) :=
\begin{cases}
\displaystyle\frac{1}{\varepsilon^{2^*-1}} |s|^{2^*-1} & \text{if } |s| \leq \varepsilon, \\
1 & \text{if } |s| > \varepsilon.
\end{cases}
\]
Using this, we define the regularized functional
\[
J_\varepsilon(u) := \frac{1}{2} \int_{\mathbb{R}^N} |\nabla^m u|^2\, dx + \int_{\mathbb{R}^N} G_-^\varepsilon(u)\, dx - \int_{\mathbb{R}^N} G_+(u)\, dx,
\]
where
\[
G_-^\varepsilon(s) := \int_0^s \varphi_\varepsilon(t) g_-(t)\, dt.
\]
By assumption (g0), there exists a constant \( c(\varepsilon) > 0 \) such that
\[
|\varphi_\varepsilon(s) g_-(s)| \leq c(\varepsilon) |s|^{2^*-1}, \quad \text{for all } s \in \mathbb{R},
\]
which implies
\[
G_-^\varepsilon(s) \leq c(\varepsilon) |s|^{2^*}.
\]
Consequently, for each \( \varepsilon \in (0,1) \), the functional \( J_\varepsilon \) is well-defined on \( \D \), is continuous, and its derivative \( J_\varepsilon'(u)(v) \) exists for all \( u \in \D \), \( v \in \mathcal{C}_0^\infty(\mathbb{R}^N) \).  
We say that a function \( u \) is a \emph{critical point} of \( J_\varepsilon \) if
\[
J_\varepsilon'(u)(v) = 0 \quad \text{for all } v \in \mathcal{C}_0^\infty(\mathbb{R}^N).
\]
\indent Now define, for \( \varepsilon \in (0,1) \),
\[
\begin{aligned}
G_\varepsilon &:= G_+ - G_-^\varepsilon, \\
\mathcal{M}_\varepsilon &:= \left\{ u \in \mathcal{D} \setminus \{0\} : \int_{\mathbb{R}^N} |\nabla^m u|^2\, dx = 2^* \int_{\mathbb{R}^N} G_\varepsilon(u)\, dx \right\}, \\
\mathcal{P}_\varepsilon &:= \left\{ u \in \D : \int_{\mathbb{R}^N} G_\varepsilon(u)\, dx > 0 \right\} \neq \emptyset, \\
c_\varepsilon &:= \inf_{u \in \mathcal{M}_\varepsilon} J_\varepsilon(u).
\end{aligned}
\]
We also introduce the scaling map \( m_{\mathcal{P}_\varepsilon} : \mathcal{P}_\varepsilon \to \mathcal{M}_\varepsilon \) given by
\[
m_{\mathcal{P}_\varepsilon}(u)(x) := u(r_\varepsilon x),
\]
where the scaling parameter is defined by
\[
r_\varepsilon(u) := \left( \frac{2^* \int_{\mathbb{R}^N} G_\varepsilon(u)\, dx}{\int_{\mathbb{R}^N} |\nabla^m u|^2\, dx} \right)^{\frac{1}{2m}}.
\]
We verify that \( m_{\mathcal{P}_\varepsilon} \) is well-defined.  
Indeed, if \( u \in \mathcal{P}_\varepsilon \), then
\[
\begin{aligned}
\int_{\mathbb{R}^N} |\nabla^m m_{\mathcal{P}_\varepsilon}(u)(x)|^2\, dx 
&= r_\varepsilon^{2m - N} \int_{\mathbb{R}^N} |\nabla^m u|^2\, dx \\
&= \left(2^* \int_{\mathbb{R}^N} G_\varepsilon(u)\, dx \right) r_\varepsilon^{-N} \\
&= 2^* \int_{\mathbb{R}^N} G_\varepsilon(m_{\mathcal{P}_\varepsilon}(u)(x))\, dx,
\end{aligned}
\]
which shows that \( m_{\mathcal{P}_\varepsilon}(u) \in \mathcal{M}_\varepsilon \).

\begin{lemma}\label{lem:theta}
	Suppose that $(u_n)\subset \cM_\eps$, $J_\eps(u_n)\to c_\eps$ and
	$$u_n\weakto\tu\neq 0\hbox{ in } \D,\;u_n(x)\to\tu(x)\quad\hbox{ for a.e. }x\in\R^N$$ for some  $\tu\in \D$. Then $u_n\to\tu$, $\tu$  is a critical point of $J_\eps$ and $J_\eps(\tu)=c_\eps$.
\end{lemma}
\begin{proof}
	The growth conditions imply that  for every $\delta>0$ there is $c(\delta)>0$ such that 
	\[
	|G_\e (u+v) - G_\e(u)| \leq \delta |u|^{2^*} + c(\delta)|v|^{2^*},\quad u,\; v \in \R,
	\]
cf. \cite[Lemma 5.1]{Med-Siem}. Let us fix an arbitrary test function $v \in \mathcal{C}_0^\infty(\mathbb{R}^N)$ and a real number $t \in \mathbb{R}$. We observe that the sequence of functions
\[
G_\varepsilon(u_n + t v) - G_\varepsilon(u_n)
\]
is uniformly integrable and tight. Therefore, by Vitali's convergence theorem, we conclude that
\[
\lim_{n \to \infty} \int_{\mathbb{R}^N} \left[ G_\varepsilon(u_n + t v) - G_\varepsilon(u_n) \right]\, dx = \int_{\mathbb{R}^N} \left[ G_\varepsilon(\tilde{u} + t v) - G_\varepsilon(\tilde{u}) \right]\, dx.
\]
Since each $u_n \in \mathcal{M}_\varepsilon$, we also have
\[
c_\varepsilon = J_\varepsilon(u_n) = \frac{1}{2} \int_{\mathbb{R}^N} |\nabla^m u_n|^2\, dx - \int_{\mathbb{R}^N} G_\varepsilon(u_n)\, dx = \left( \frac{2^*}{2} - 1 \right) \int_{\mathbb{R}^N} G_\varepsilon(u_n)\, dx.
\]
Taking the limit as $n \to \infty$, we obtain
\begin{equation} \label{eq:1.34}
A := \lim_{n \to \infty} \int_{\mathbb{R}^N} G_\varepsilon(u_n)\, dx = \frac{1}{2^*} \left( \frac{1}{2} - \frac{1}{2^*} \right)^{-1} c_\varepsilon > 0.
\end{equation}
Combining the above with Vitali’s theorem yields:
\begin{equation} \label{eq:1.37}
\begin{aligned}
\lim_{n \to \infty} \int_{\mathbb{R}^N} G_\varepsilon(u_n + t v)\, dx 
&= \lim_{n \to \infty} \int_{\mathbb{R}^N} G_\varepsilon(u_n)\, dx + \int_{\mathbb{R}^N} G_\varepsilon(\tilde{u} + t v)\, dx - \int_{\mathbb{R}^N} G_\varepsilon(\tilde{u})\, dx \\
&= A + \int_{\mathbb{R}^N} G_\varepsilon(\tilde{u} + t v)\, dx - \int_{\mathbb{R}^N} G_\varepsilon(\tilde{u})\, dx.
\end{aligned}
\end{equation}
From \eqref{eq:1.34}, we know that $u_n + t v \in \mathcal{P}_\varepsilon$ for all sufficiently large $n$, provided that $|t|$ is small enough. Hence, using \eqref{eq:1.37}, we get
\[
\begin{aligned}
\lim_{n \to \infty} \frac{1}{t} \left[ \left( \int_{\mathbb{R}^N} G_\varepsilon(u_n + t v)\, dx \right)^{\frac{N - 2m}{N}} 
- \left( \int_{\mathbb{R}^N} G_\varepsilon(u_n)\, dx \right)^{\frac{N - 2m}{N}} \right] 
= \\
\frac{1}{t} \left[ \left( A + \int_{\mathbb{R}^N} G_\varepsilon(\tilde{u} + t v)\, dx - \int_{\mathbb{R}^N} G_\varepsilon(\tilde{u})\, dx \right)^{\frac{N - 2m}{N}} - A^{\frac{N - 2m}{N}} \right].
\end{aligned}
\]
By the Lebesgue dominated convergence theorem and differentiability of the power function, we conclude:
\begin{equation} \label{eq:thetain1}
\lim_{t \to 0} \frac{1}{t} \left[ \left( A + \int_{\mathbb{R}^N} G_\varepsilon(\tilde{u} + t v)\, dx - \int_{\mathbb{R}^N} G_\varepsilon(\tilde{u})\, dx \right)^{\frac{N - 2m}{N}} - A^{\frac{N - 2m}{N}} \right] 
= \frac{N - 2m}{N} A^{-\frac{2m}{N}} \int_{\mathbb{R}^N} g_\varepsilon(\tilde{u}) v\, dx,
\end{equation}
where \( g_\varepsilon := G_\varepsilon' = g_+ - \varphi_\varepsilon g_- \).
If \( u_n + t v \in \mathcal{P}_\varepsilon \), then by the definition of \( m_{\mathcal{P}_\varepsilon} \), we have \( J_\varepsilon(m_{\mathcal{P}_\varepsilon}(u_n + t v)) \geq c_\varepsilon \), which implies
\[
r_\varepsilon^{2m - N} \left( \frac{1}{2} - \frac{1}{2^*} \right) \int_{\mathbb{R}^N} |\nabla^m (u_n + t v)|^2\, dx \geq c_\varepsilon.
\]
Raising both sides to the power \( \frac{2m}{N} \) and using the definition of \( r_\varepsilon \), we derive the inequality
\begin{equation} \label{eq:1.35}
\left( \frac{1}{2} - \frac{1}{2^*} \right)^{\frac{2m}{N}} \int_{\mathbb{R}^N} |\nabla^m (u_n + t v)|^2\, dx \geq c_\varepsilon^{\frac{2m}{N}} \left( 2^* \int_{\mathbb{R}^N} G_\varepsilon(u_n + t v)\, dx \right)^{\frac{N - 2m}{N}}.
\end{equation}
In addition, since \( u_n \in \mathcal{M}_\varepsilon \) and \( J_\varepsilon(u_n) \to c_\varepsilon \), it follows that
\begin{equation} \label{eq:1.36}
\int_{\mathbb{R}^N} |\nabla^m u_n|^2\, dx \to \left( \frac{1}{2} - \frac{1}{2^*} \right)^{-1} c_\varepsilon.
\end{equation}
We now expand the squared gradient:
\[
\int_{\mathbb{R}^N} |\nabla^m(u_n + t v)|^2\, dx = \int_{\mathbb{R}^N} |\nabla^m u_n|^2\, dx + t^2 \int_{\mathbb{R}^N} |\nabla^m v|^2\, dx + 2t \int_{\mathbb{R}^N} \nabla^m u_n \cdot \nabla^m v\, dx.
\]
Thus,
\[
\frac{1}{2t} \left[ \int_{\mathbb{R}^N} |\nabla^m(u_n + t v)|^2\, dx - \int_{\mathbb{R}^N} |\nabla^m u_n|^2\, dx \right] = \int_{\mathbb{R}^N} \nabla^m u_n \cdot \nabla^m v\, dx + \frac{t}{2} \int_{\mathbb{R}^N} |\nabla^m v|^2\, dx.
\]
Using \eqref{eq:1.35} and the fact that \( u_n \in \mathcal{M}_\varepsilon \), for small \( t > 0 \), we obtain
\[
\begin{aligned}
\int_{\mathbb{R}^N} \nabla^m u_n \cdot \nabla^m v\, dx + \frac{t}{2} \int_{\mathbb{R}^N} |\nabla^m v|^2\, dx 
&\geq \frac{1}{2t} \left( c_\varepsilon^{\frac{2m}{N}} \left( \frac{1}{2} - \frac{1}{2^*} \right)^{-\frac{2m}{N}} \right) \\
&\quad \times \left[ \left( 2^* \int_{\mathbb{R}^N} G_\varepsilon(u_n + t v)\, dx \right)^{\frac{N - 2m}{N}} - \left( 2^* \int_{\mathbb{R}^N} G_\varepsilon(u_n)\, dx \right)^{\frac{N - 2m}{N}} \right].
\end{aligned}
\]
Passing to the limit as \( n \to \infty \), and using \eqref{eq:1.37}, \eqref{eq:1.34}, and \eqref{eq:1.36}, we deduce that for sufficiently small \( t > 0 \),
\[
\begin{aligned}
\int_{\mathbb{R}^N} \nabla^m \tilde{u} \cdot \nabla^m v\, dx + \frac{t}{2} \int_{\mathbb{R}^N} |\nabla^m v|^2\, dx 
\geq \frac{2^*}{2} A^{\frac{2m}{N}} \frac{1}{t} \left[ \left( A + \int_{\mathbb{R}^N} G_\varepsilon(\tilde{u} + t v)\, dx - \int_{\mathbb{R}^N} G_\varepsilon(\tilde{u})\, dx \right)^{\frac{N - 2m}{N}} - A^{\frac{N - 2m}{N}} \right].
\end{aligned}
\]
Letting \( t \to 0^+ \) and applying \eqref{eq:thetain1}, we conclude
\[
\int_{\mathbb{R}^N} \nabla^m \tilde{u} \cdot \nabla^m v\, dx \geq \int_{\mathbb{R}^N} g_\varepsilon(\tilde{u}) v\, dx.
\]
Since \( v \in \mathcal{C}_0^\infty(\mathbb{R}^N) \) was arbitrary, we infer that \( \tilde{u} \) is a weak solution (i.e., a critical point) of \( J_\varepsilon \).
Applying the Poho\v{z}aev identity (Theorem~\ref{th:Poho}) to the equation \( (-\Delta)^m u = g_\varepsilon(u) \) with \( G_\varepsilon \in L^1(\mathbb{R}^N) \), we deduce that \( \tilde{u} \in \mathcal{M}_\varepsilon \), which implies
\[
c_\varepsilon \leq J_\varepsilon(\tilde{u}) = \left( \frac{1}{2} - \frac{1}{2^*} \right) \int_{\mathbb{R}^N} |\nabla^m \tilde{u}|^2\, dx \leq \liminf_{n \to \infty} \left( \frac{1}{2} - \frac{1}{2^*} \right) \int_{\mathbb{R}^N} |\nabla^m u_n|^2\, dx = c_\varepsilon,
\]
where the weak lower semicontinuity of the norm was used. Therefore, \( J_\varepsilon(\tilde{u}) = c_\varepsilon \) and \( \|u_n\| \to \|\tilde{u}\| \), which yields strong convergence \( u_n \to \tilde{u} \) in \( \D\).
\end{proof}

\begin{proof}[Proof of Theorem \ref{thm:2}]
	Take a minimizing sequence $(u_n)$ in $\cM_\e$ of $J_\e$, i.e., $J_\e(u_n)\to c_\e$.
	Since $u_n\in \cM_\e$, $n\geq 1$, we have 
	\[
	J_\e(u_n) = \frac{1}{2}\int_{\R^N} |\nabla^m u_n|^2\, dx-\int_{\R^N} G_{\varepsilon}(u_n)\, dx=\le \frac{1}{2}-\frac{1}{2^*}\pr\int_{\R^N}|\nabla^m u_n|^2\d x \to c_\e,
	\]
	and so $(u_n)$ is bounded in $\D$.
	Moreover, we have 
	\[
	2^*\int_{\R^N}G _+(u_n)\d x \geq 	2^*\int_{\R^N}G _{\varepsilon}(u_n)\d x=\int_{\R^N}|\nabla^m u_n|^2\d x\to \le \frac{1}{2}-\frac{1}{2^*}\pr^{-1} c_\e.
	\]
	By the assumption $G_+$ satisfies \eqref{eq:Psi} and by the above inequality $\int_{\R^N}G_+(u_n)\, dx \not\to 0$ so by Lemma \ref{lem:Lions}  \eqref{eq:LionsCond11} is not satisfied.
	For the Lemma of Lions in this case we can take a subsequence, we may choose $(y_n)$ in $\R^N$ and $0\neq u_\e \in \D$ such that
	\[
	u_n(\cdot + y_n) \rightharpoonup u_\e \quad\text{in }\D,\quad u_n(x+y_n) \to u_\e(x) \quad \text{for a.e. }x\in \R^N,
	\]
	as $n\to \infty$.
	In view of Lemma \ref{lem:theta}, $u_\e\in \cM_\e$ is a critical point of $J_\e$ at the level $c_\e$.
	
	\noindent Choose $\e_n\to 0^+$.
	Fix an arbitrary $u\in \cM$.
	Since $G_\e(s) \geq G(s)$, for all $s\in \R$ and $\e \in (0,1)$, we deduce that
	\[
	\int_{\R^N}G_{\e_n}(u)\d x \geq\int_{\R^N}G(u)\d x = \frac{1}{2^*}\int_{\R^N}|\nabla^m u|^2 \d x >0,
	\]
	so $u \in \cP_{\e_n}$ and $m_{\cP_{\e_n}}(u)\in \cM_{\e_n}$ is well-defined. Since $u_{\e_n}$ is the infimum of $J_{\e_n}$ in $\cM_{\e_n}$, we have 
	\[
	\begin{aligned}
		J_{\e_n}(u_{\e_n}) &\leq J_{\e_n}(m_{\cP_{\e_n}}(u)) = \le \frac{1}{2}-\frac{1}{2^*}\pr \underbrace{\le\frac{2^*\int_{\R^N}G_{\e_n}(u)\d x}{\int_{\R^N}|\nabla^m u|^2 \d x}\pr^\frac{2m-N}{2m}}_{r_{\e_n}(u)^{4-N}}\int_{\R^N}|\nabla^m u|^2\d x\\
		&=\le\frac{1}{2} - \frac{1}{2^*} \pr\le \int_{\R^N}|\nabla^m u|^2\d x\pr^\frac{N}{2m}\le 2^*\int_{\R^N}G_{\e_n}(u)\d x \pr ^{-\frac{N-2m}{2m}}\\
		&\leq \le\frac{1}{2} - \frac{1}{2^*} \pr\le \int_{\R^N}|\nabla^m u|^2\d x\pr^\frac{N}{2m}\underbrace{\le 2^*\int_{\R^N}G(u)\d x \pr ^{-\frac{N-2m}{2m}}}_{=\le \int_{\R^N}|\nabla^m u |^2\d x \pr ^{-\frac{N-2m}{2m}}}\\
		&= \le\frac{1}{2} - \frac{1}{2^*} \pr \int_{\R^N}|\nabla^m u|^2\d x = J(u).
	\end{aligned}
	\]
	Thus $J_{\e_n} ( u_{\e_n})  \leq \inf_\cM J$ and 
	\begin{equation}\label{eq:1.38}
		\int_{\R^N}|\nabla^m u_{\e_n}|^2 \d x \leq \le \frac{1}{2}- \frac{1}{2^*}\pr^{-1} \inf_\cM J,\qquad\text{for every }n.
	\end{equation}
	We have $G_\e(s) \leq G_{1/2}(s)$, for all $s\in \R$ and $\e\in (0,1/2)$, so 
	\[
	\int_{\R^N} G_{1/2}(u_\e)\d x \geq \int_{\R^N} G_\e(u_\e)\d x = \frac{1}{2^*}\int_{\R^N} |\nabla^m u_\e|^2\d x> 0\implies u_\e \in \cP_{1/2},
	\]
	and some calculations yield
	\[
	J_\e(u_\e) \geq J_{1/2}(m_{\cP_{1/2}}(u_\e))\geq J_{1/2}(u_{1/2}).
	\]
	Therefore, we get
	\[
	2^* \int_{\R^N}G_+(u_{\e_n})\d x \geq \int_{\R^N}|\nabla^m u_{\e_n}|^2\d x  = \le \frac{1}{2}-\frac{1}{2^*}\pr ^{-1}J_{\e_n}(u_{\e_n})\geq \le \frac{1}{2}-\frac{1}{2^*}\pr ^{-1} J_{1/2}(u_{1/2}) >0.
	\]
	By \eqref{eq:1.38}, $(u_{\e_n})$ is bounded in $\D$ and $\int_{\R^N} G_+(u_{\e_n})\d x>c>0$, for some constant $c$.
	In view of Lemma \ref{lem:Lions}, \eqref{eq:LionsCond11} is not satisfied.
	Passing to a subsequence, there is $(y_n)$ in $\R^N$ such that $u_{\e_n}(\cdot + y_n)\rightharpoonup u_0\neq 0$ and $u_{\e_n}(x+y_n) \to u_0(x)$ a.e. in $\R^N$.
	We write $\tu_n:= u_{\e_n}(\cdot+y_n)$ for short.
	Since $g_-$ is continuous and $g_-(0)=0$, one may check that, for every $v\in C^\infty_0(\R^N)$,
	\[
	\left | \frac{1}{\e_n^{2\gg-1}}|\tu_n|^{2\gg-1}\chi_{\left\{|\tu_n|\leq \e_n\right\}} g_-(\tu_n)v \right|\leq \left|\chi_{\left\{|\tu_n|\leq \e_n\right\}} g_-(\tu_n)v  \right|  \to 0 \quad\text{a.e. in }\R^N
	\] 
	and 
	\[
	\left | \chi_{\left\{|\tu_n|> \e_n\right\}} g_-(\tu_n)v - g_-(u_0)v  \right| \to 0 \quad\text{a.e. in }\R^N.
	\]
	Due to the estimate $|g_-(\tu_n)v|\leq c\le 1 + |\tu_n|^{2^*-1}\pr |v|$, the family $\{g_-(\tu_n)v\}$ is uniformly integrable (and tight because of the compact support).
	In view of Vitali's convergence theorem
	\[
	\begin{aligned}
		&\int_{\R^N} \left | \f_{\e_n}(\tu_n)g_-(\tu_n)v - g_-(u_0)v\right | \d x \\
		&\quad \leq \int_{\R^N}\left | \frac{1}{\e_n^{2^*-1}}|\tu_n|^{2^*-1}\chi_{\left\{|\tu_n|\leq \e_n\right\}} g_-(\tu_n)v \right|\d x + \int_{\R^N}\left | \chi_{\left\{|\tu_n|> \e_n\right\}} g_-(\tu_n)v - g_-(u_0)v  \right|\d x\to 0,
	\end{aligned}
	\]
	as $n\to \infty$.
	Similarly, we obtain
	\[
	\int_{\R^N} g_+(\tu_n)v \d x \to \int_{\R^N} g_+(u_0)v \d x.
	\]
	Gathering the above we deduce that
	\[
	\begin{aligned}
		J_{\e_n}^\prime(\tu_n)(v) &= \int_{\R^N}\nabla^m\tu_n \nabla^m v \d x - \int_{\R^N}g_+(\tu_n)v \d x + \int_{\R^N}\f_{\e_n}(\tu_n) g_-(\tu_n) v\d x\\
		&\to \int_{\R^N}\nabla^m u_0 \nabla^m v \d x - \int_{\R^N}g_+(u_0)v\d x + \int_{\R^N} g_-(u_0)v\d x.
	\end{aligned}
	\]
	Each $\tu_n$ is  a critical point of $J_{\e_n}$, since so is $u_{\e_n}$ (translation invariance), hence
	\[
	\int_{\R^N}\nabla^m u_0 \nabla^m v \d x =  \int_{\R^N}g(u_0)v\d x,
	\]
	i.e., $u_0$ is a weak solution to \eqref{problem}.
	By Lebesgue's dominated convergence theorem one may show that
	\[
	G^{\e_n}_-(\tu_n) \to G_-(u_0) \quad \text{a.e. in }\R^N,
	\]
	as $n\to\infty$, and, on the other hand, 
	\[
	2^* \int_{\R^N}G^{\e_n}_-(\tu_n)\d x = 2^* \int_{\R^N}G_+(\tu_n)\d x - \int_{\R^N}|\nabla^m \tu_n|^2\d x\leq c\le \sup_{n\geq 1}\|\tu_n\|_{\D}\pr <\infty,
	\]
	where we used the fact that $\tu_n\in \cM_{\e_n}$ and \eqref{eq:1.38}.
	By Fatou's lemma and by the above
	\[
	\int_{\R^N}G_-(u_0)\d x\leq \liminf_{n\to\infty}\int_{\R^N}G^{\e_n}_-(\tu_n)\d x < \infty,
	\]
	namely, we have shown that $G_-(u_0)\in L^1(\R^N)$.
	By the Poho\v{z}aev identity, we infer that $u_0\in \cM$.
	Lastly, we show that $J(u_0) = \inf_\cM J$.
	We use the weak l.s.c. of the norm and \eqref{eq:1.38} to find that
	\[
	\begin{aligned}
		J(u_0) &= \le \frac{1}{2}- \frac{1}{2^*}\pr\int_{\R^N}|\nabla^m u_0|^2\d x\leq \liminf_{n\to\infty} \le \frac{1}{2}- \frac{1}{2^*}\pr\int_{\R^N}|\nabla^m \tu_n|^2\d x\\
		&= \liminf_{n\to\infty} \le \frac{1}{2}- \frac{1}{2^*}\pr\int_{\R^N}|\nabla^m u_{\e_n}|^2\d x\leq \inf_\cM J.
	\end{aligned}
	\]
\end{proof}

 \section{Polyharmonic Logarithmic Sobolev inequality}\label{sec:Log}
\begin{lemma}\label{Poly-inequality}
Let $m \in \N$, $m\geq1$ and  \( u \in \mathcal{D}^{m+1,2}(\R^N) \) such that $\int_{\R^N} |u|^2\, dx=1$. Then
\begin{equation} \label{Poly-inequality}
\left( \int_{\mathbb{R}^N} |\nabla^m  u|^2 \, dx \right)^{\frac{1}{m}} 
< \Big( \int_{\R^N} |\nabla^{m+1} u|^2\, dx \Big)^{\frac{1}{m+1}}
\end{equation}

\end{lemma}
\begin{proof}
    Let us define the Fourier transform $\hat{u}$ of $u$ (whenever
possible) as
\[
\hat{u}(\xi) = \frac{1}{(2\pi)^{N/2}} \int_{\mathbb{R}^N} e^{-i x \cdot \xi} u(x) \, dx, \quad \xi \in \mathbb{R}^N.
\]
If $u \in \mathcal{D}^{m,2}(\R^N)$ and 
\[
\int_{\mathbb{R}^N} |u|^2 \, dx = 1,
\]
then $u \in W^{m,2}(\mathbb{R}^N)$, and by the Plancherel theorem
\[
\| u \|_{L^2(\mathbb{R}^N)} = \| \hat{u} \|_{L^2(\mathbb{R}^N)},
\]
\[
\| \nabla u \|_{L^2(\mathbb{R}^N)} = \| \reallywidehat{\nabla u} \|_{L^2(\mathbb{R}^N)} = \| \xi \hat{u} \|_{L^2(\mathbb{R}^N)},
\]
\[
\| \Delta u \|_{L^2(\mathbb{R}^N)} = \|\reallywidehat{ \Delta u} \|_{L^2(\mathbb{R}^N)} = \| |\xi|^2 \hat{u} \|_{L^2(\mathbb{R}^N)}.
\]\[
\| \nabla^m u \|_{L^2(\mathbb{R}^N)} = \| \mathrm{\reallywidehat{\nabla^m u }} \|_{L^2(\mathbb{R}^N)} = \| |\xi|^{m}\hat{u} \|_{L^2(\mathbb{R}^N)}.
\] By induction we have \eqref{Poly-inequality}. Indeed, for $m=1$ it is true by \cite[Lemma 5.1]{Med-Siem}. Then suppose true for $m$, we show that it is true for $m+1$.\bigskip
\noindent 
$$
\begin{aligned}
    \Big(\int_{\R^N} ||\xi|^{m+1} \hat u|^2 \, d\xi\Big)^{\frac{1}{m+1}}& \leq  \Big(\int_{\R^N} ||\xi|^m\hat u|^2\, d\xi  \Big)^{\frac{1}{2(m+1)}}\Big( \int_{\R^N} ||\xi|^{m+2} \hat u|^2\, d\xi\Big)^{\frac{1}{2(m+1)}}\\ &\leq \Big(\int_{\R^N} ||\xi|^{m+1}\hat u|^2\, d\xi  \Big)^{\frac{m}{2(m+1)^2}}\Big( \int_{\R^N} ||\xi|^{m+2} \hat u|^2\, d\xi\Big)^{\frac{1}{2(m+1)}}.
\end{aligned}
$$

$$
\begin{aligned}
    \Big(\int_{\R^N} ||\xi|^{m+1} \hat u|^2 \, d\xi\Big)^{\frac{1}{m+1}-\frac{m}{2(m+1)^2}}\leq \Big( \int_{\R^N} ||\xi|^{m+2} \hat u|^2\, d\xi\Big)^{\frac{1}{2(m+1)}}. 
\end{aligned}
$$
Finally we have 
$$
\begin{aligned}
    \Big(\int_{\R^N} ||\xi|^{m+1} \hat u|^2 \, d\xi\Big)^{\frac{1}{m+1}}\leq \Big( \int_{\R^N} ||\xi|^{m+2} \hat u|^2\, d\xi\Big)^{\frac{1}{m+2}}. 
\end{aligned}
$$

\end{proof}

\begin{proof}[Proof of Theorem \ref{th1.4}]
We show that the following inequality is satisfied
\begin{equation}
    \left(\int_{\mathbb{R}^N} |\nabla^m u|^2 \, dx\right)^{\frac{N}{N-2m}} \geq C_{N, \log} \int_{\mathbb{R}^N} |u|^2 \log |u| \, dx, \quad \text{for any } u \in \D, \label{eq:6.1}
\end{equation}
where
\begin{equation}
C_{N, \log} = 2^{*} \left( \frac{1}{2} - \frac{1}{2^{*}} \right)^{- \frac{2m}{N-2m}} \inf_{\mathcal{M}} J^{\frac{2m}{N-2m}}.
\end{equation}
Indeed, it is enough to consider \( u \in\D \) such that \( \int_{\mathbb{R}^N} |u|^2 \log |u| \, dx > 0 \). We then obtain \( u(r) \in \mathcal{M}, \) where
\begin{equation}
r := \left(\frac{2^{*} \int_{\mathbb{R}^N} |u|^2 \log |u| \, dx}{\int_{\mathbb{R}^N} |\nabla^m u|^2} \right)^{\frac{1}{2m}}.
\end{equation}
Indeed 
\begin{equation*}\begin{aligned}
    & \int_{\mathbb{R}^N} |\nabla^m u(rx)|^2\, dr=r^{2m-N}\int_{\mathbb{R}^N} |\nabla^m u|^2 \, dx = \left(\frac{2^{*} \int_{\mathbb{R}^N} |u|^2 \log |u| \, dx}{\int_{\mathbb{R}^N} |\nabla^m u|^2} \right)^{\frac{2m-N}{2m}}\int_{\mathbb{R}^N} |\nabla^m u|^2 \, dx \\ & = 2^{*} \int_{\mathbb{R}^N} |u|^2 \log |u| \, dx\bigg( \frac{\int_{\mathbb{R}^N} |\nabla^m u|^2 \, dx}{2^{*} \int_{\mathbb{R}^N} |u|^2 \log |u| \, dx} \bigg)^{1-\frac{2m-N}{2m}} \\ & =2^{*} \int_{\mathbb{R}^N} |u|^2 \log |u| \, dx\bigg( \frac{\int_{\mathbb{R}^N} |\nabla^m u|^2 \, dx}{2^{*} \int_{\mathbb{R}^N} |u|^2 \log |u| \, dx} \bigg)^{1-\frac{N}{2m}} \\ &= 2^{*} \int_{\mathbb{R}^N} |u|^2 \log |u| \, r^{-N}\,  dx =2^{*} \int_{\mathbb{R}^N} |u(rx)|^2 \log |u(xr)| \, dx.
\end{aligned}
\end{equation*}

Hence \( J(u(r)) \geq \inf_{\mathcal{M}} J \) and therefore
\begin{equation*}
\begin{aligned}
    & \frac{1}{2} \int_{\mathbb{R}^N} |\nabla^m u(rx)| dx- \int_{\mathbb{R}^N} |u(rx)|^2 \log |u(rx)| \, dx=\bigg( \frac{1}{2}-\frac{1}{2^{*}}\bigg) r^{2m-N} \int_{\mathbb{R}^N} |\nabla^m u|^2 \, dx\\ &= \bigg( \frac{1}{2}-\frac{1}{2^{*}}\bigg)\left(\frac{2^{*} \int_{\mathbb{R}^N} |u|^2 \log |u| \, dx}{\int_{\mathbb{R}^N} |\nabla^m u|^2} \right)^{\frac{2m-N}{2m}} \int_{\mathbb{R}^N} |\nabla^m u|^2 \, dx\\ & = \bigg( \frac{1}{2}-\frac{1}{2^{*}}\bigg) \bigg(2^{*} \int_{\mathbb{R}^N} |u|^2 \log |u| \, dx\bigg)^{\frac{2m-N}{N}} \bigg(\int_{\mathbb{R}^N} |\nabla^m u|^2 \, dx\bigg)^{\frac{N}{2m}}\geq (\inf\limits_{\mathcal{M}} J).
\end{aligned}
\end{equation*}
\begin{equation*}
\begin{aligned}
& \bigg(\int_{\mathbb{R}^N} |\nabla^m u|^2 \, dx\bigg)^{\frac{N}{2m}}\geq \bigg( \frac{1}{2}-\frac{1}{2^{*}}\bigg)^{-1} \bigg(2^{*} \int_{\mathbb{R}^N} |u|^2 \log |u| \, dx\bigg)^{-\frac{2m-N}{N}}(\inf\limits_{\mathcal{M}} J) \\ &\bigg(\int_{\mathbb{R}^N} |\nabla^m u|^2 \, dx\bigg)^{\frac{N}{N-2m}}\geq \bigg(\frac{1}{2}-\frac{1}{2^{*}}   \bigg)^{-\frac{N-2m}{2m}} 2^{*} \int_{\mathbb{R}^N} |u|^2 \log |u| \, dx \, (\inf\limits_{\mathcal{M}} J)^{\frac{2m}{N-2m}}
    \end{aligned}
    \end{equation*}
and we get \eqref{eq:6.1}.
Now note that \eqref{eq:6.1} is equivalent to
\begin{equation}\label{eq:6.2}
    \left(\int_{\mathbb{R}^N} |\nabla^m u|^2 \, dx\right)^{\frac{N}{N-2m}} \geq C_{N, \log} \max_{\alpha \in \mathbb{R}} \left\{ e^{- \alpha 2^{*}} \int_{\mathbb{R}^N} |e^\alpha u|^2 \log |e^\alpha u| \, dx \right\}, \quad \text{for } u \in D^{m,2}(\mathbb{R}^N). 
\end{equation}

Assuming that \( \int_{\mathbb{R}^N} |u|^2 \, dx = 1 \), the maximum of the right-hand side of \eqref{eq:6.2} is attained at 
\[
\alpha_1 = \frac{N-2m}{4m} -  \int_{\mathbb{R}^N} |u|^2 \log |u| \, dx.
\] Indeed $f(\alpha):= e^{(2-2^{*})\alpha}(\alpha \int_{\mathbb{R}^N} |u|^2\, dx + \int_{\mathbb{R}^N} |u|^2 \log |u| \, dx )$ and we computed $$f'(\alpha)=e^{(2-2^{*})\alpha}((2-2^{*})\alpha+ (2-2^{*})\int_{\mathbb{R}^N} |u|^2\log |u|\, dx + 1 ).$$
Hence we get
\begin{equation*}\begin{aligned}
 & \frac{N}{N-2m} \log \left(\int_{\mathbb{R}^N} |\nabla^m u|^2 \, dx \right) \geq \log(C_{N, \log}) - \alpha_1 2^{*} + 2 \alpha_1 + \log\left(\alpha_1  + \int_{\mathbb{R}^N} |u|^2 \log |u| \, dx\right) \\ & \frac{N}{N-2m} \log \left(\int_{\mathbb{R}^N} |\nabla^m u|^2 \, dx \right) \geq \log(C_{N, \log}) + (2-2^{*}) \alpha_1 + \log\left(\frac{N-2m}{4m}\right)
\end{aligned}
\end{equation*}
by definition of $\alpha_1$ and $2^{*}$ we have \begin{equation*}
\begin{aligned} \frac{N}{N-2m} \log \left(\int_{\mathbb{R}^N} |\nabla^m u|^2 \, dx \right) \geq &\log(C_{N, \log}) -\frac{4m}{N-2m}\bigg(\frac{N-2m}{4m}-\int_{\mathbb{R}^N} |u|^2 \log |u|\, dx\bigg) \\ & + \log\left(\frac{N-2m}{4m}\right) 
\end{aligned}
\end{equation*}
That is,
\begin{equation}
\frac{N}{N-2m} \log \left(\int_{\mathbb{R}^N} |\nabla^m u|^2 \, dx \right) \geq \log\left( C_{N, \log} \frac{N-2m}{4m} e^{-1} \right) + \frac{4m}{N-2m} \int_{\mathbb{R}^N} |u|^2 \log |u| \, dx.
\end{equation}

Thus,
\begin{equation}\label{eq:1.7}
\frac{N}{4m} \log \int_{\mathbb{R}^N} |\nabla^m u|^2 \, dx  \geq \frac{N-2m}{4m} \log \left( C_{N, \log} \frac{N-2m}{4m} e^{-1} \right) + \int_{\mathbb{R}^N} |u|^2 \log |u| \, dx.
\end{equation}
\begin{equation*}\begin{aligned}
    & \log \bigg(\frac{\bigg(\int_{\mathbb{R}^N} |\nabla^m u|^2 \, dx\bigg)^{\frac{N}{4m}}}{\bigg(C_{N,\log}\frac{N-2m}{4me}\bigg)^{\frac{N-2m}{4m}}} \bigg)  \geq \int_{\mathbb{R}^N} |u|^2 \log |u| \, dx \\ &\frac{N}{4m} \log\bigg( \bigg(  \frac{4m e}{C_{N,\log}(N-2m)}   \bigg)^{\frac{N-2m}{N}} \int_{\mathbb{R}^N} |\nabla^mu |^2\, dx \bigg) \geq   \int_{\mathbb{R}^N} |u|^2 \log |u| \, dx
\end{aligned}
\end{equation*}
Thus \eqref{eq:6.1} holds.
\\
We show that the constant in \eqref{eq: 1.7} is optimal, that is, there is \( u \in \D \) such that the equality holds. First of all, notice that if \( u_0 \) is a minimizer given by Theorem \ref{thm:2}, then for \( u_0 \) we have the equality in \eqref{eq:6.1}:
\begin{equation}\label{eq: id u_0}
    \left(\int_{\mathbb{R}^N} |\nabla^m u_0|^2 \, dx \right)^{\frac{N}{N-2m}} = C_{N, \log} \int_{\mathbb{R}^N} |u_0|^2 \log |u_0| \, dx. 
\end{equation}

We use \eqref{eq:6.1} for the family of functions \( \frac{e^{\alpha }}{\|u_0\|_{L^2}} u_0 \in\D, \, \alpha \in \mathbb{R}, \) to get
\begin{equation}
    \left(\int_{\mathbb{R}^N} |\nabla^m u_0|^2 \, dx\right)^{\frac{N}{N-2m}} \geq C_{N, \log} \|u_0\|_{L^2}^{2^{*}-2} e^{ (2^{*} - 2) \alpha} \int_{\mathbb{R}^N} |u_0|^2 \log \left|\frac{e^{\alpha}}{\|u_0\|_{L^2} } u_0 \right| \, dx, \quad \alpha \in \mathbb{R}. \label{eq:6.4}
\end{equation}
Now let us consider the function \( f: \mathbb{R} \rightarrow \mathbb{R} \) given by
\begin{equation}
f(\alpha) := C_{N, \log} \|u_0\|_{L^2}^{(2^{*}-2)} e^{(2- 2^{*} ) \alpha} \int_{\mathbb{R}^N} |u_0|^2 \log \left| \frac{e^{\alpha}}{\|u_0\|_{L^2} } u_0 \right| \, dx - \left(\int_{\mathbb{R}^N} |\nabla^m u_0|^2 \, dx\right)^{\frac{N}{N-2m}}.
\end{equation}
Computing the first derivative \begin{equation*}\begin{aligned}
   & f'(\alpha)= C_{N, \log} \|u_0\|_{L^2}^{(2^{*}-2)} e^{(2-2^{*})\alpha}\bigg((2-2^{*}) \bigg(\alpha \int_{\mathbb{R}^N} |u_0|^2 \, dx - \int_{\mathbb{R}^N} |u_0|^2\log|\frac{u_0}{\|u_0\|_{L^2}}|     \bigg) +\int_{\mathbb{R}^N} |u_0|^2 \, dx \bigg) \\ & =  C_{N, \log} \|u_0\|_{L^2}^{2^{*}} e^{(2-2^{*})}\bigg((2-2^{*})\bigg(\alpha- \int_{\mathbb{R}^N} |\frac{u_0}{\|u_0\|_{L^2}}|^2\log|\frac{u_0}{\|u_0\|_{L^2}}|     \bigg) +1 \bigg),
\end{aligned}\end{equation*} 
and imposed the condition $f'(\alpha)\geq 0$, we note that $f$ has only one maximum point
\begin{equation}
 \alpha_2 = \frac{N-2m}{4m} -  \int_{\mathbb{R}^N} \bigg|\frac{u_0}{\|u_0\|_{L^2}}\bigg|^2 \log \left|\frac{u_0}{\|u_0\|_{L^2}}\right| \, dx.
\end{equation}

On the other hand, \( f \) attains maximum at \( \alpha = \log(\|u_0\|_{L^2}) \) in view of \eqref{eq:6.4} and \eqref{eq: id u_0}, thus
\begin{equation*}\begin{aligned}
    & \int_{\mathbb{R}^N} \bigg|\frac{u_0}{\|u_0\|_{L^2}}\bigg|^2 \log \left|\frac{u_0}{\|u_0\|_{L^2}}\right| \, dx = \frac{N-2m}{4m} - \log(\|u_0\|_{L^2}) \\& \int_{\mathbb{R}^N} \bigg|\frac{u_0}{\|u_0\|_{L^2}}\bigg|^2 \log \left|u_0\right| \, dx - \int_{\mathbb{R}^N} \bigg|\frac{u_0}{\|u_0\|_{L^2}}\bigg|^2 \log \|u_0\|_{L^2} = \frac{N-2m}{4m} - \log(\|u_0\|_{L^2})
\end{aligned}\end{equation*}since $u_0 \in \mathcal{M}$
\begin{equation*}\begin{aligned}
    &\frac{1}{2^{*}\|u_0\|_{L^2}} \int_{\mathbb{R}^N} |\nabla^mu_0|^2 \, dx= \frac{N-2m}{4m} \\ & \frac{1}{\|u_0\|_{L^2}} \int_{\mathbb{R}^N} |\nabla^mu_0|^2 \, dx=\frac{2N}{N-2m} \frac{N-2m}{4m}
\end{aligned}\end{equation*}
\begin{equation}
\frac{1}{\|u_0\|_{L^2}} \int_{\mathbb{R}^N} |\nabla^mu_0|^2 \, dx= \frac{N}{2m}.
\end{equation} We note that 
\begin{equation}\label{f=0}
    f(\log \| u_0\|_{L^2})=0.
\end{equation} Indeed by \eqref{eq: id u_0}
\begin{equation*}
\begin{aligned}
    & f(\log \| u_0\|_{L^2})\\ &=
    C_{N, \log} \|u_0\|_{L^2}^{(2^{*}-2)} e^{(2- 2^{*} ) \log \| u_0\|_{L^2}}  
    \int_{\mathbb{R}^N} |u_0|^2 \log \left| \frac{e^{\log \| u_0\|_{L^2}}}{\|u_0\|_{L^2} } u_0 \right| \, dx - \left(\int_{\mathbb{R}^N} |\nabla^m u_0|^2 \, dx\right)^{\frac{N}{N-2m}} \\
    &=C_{N,\log} \int_{\mathbb{R}^N} |u_0|^2 \log \left|  u_0 \right| \, dx - \left(\int_{\mathbb{R}^N} |\nabla^m u_0|^2 \, dx\right)^{\frac{N}{N-2m}}=0
\end{aligned}
\end{equation*}
By \eqref{f=0} and the unicity of the maximum we have
\begin{equation*}\begin{aligned}
    & f(\alpha_2)= 0 \\ & C_{N, \log} \|u_0\|_{L^2}^{(2^{*}-2)} e^{(2- 2^{*} ) \alpha_2} \int_{\mathbb{R}^N} |u_0|^2 \log \left| \frac{e^{\alpha_2}}{\|u_0\|_{L^2} } u_0 \right| \, dx - \left(\int_{\mathbb{R}^N} |\nabla^m u_0|^2 \, dx\right)^{\frac{N}{N-2m}}=0\\ &
   \left(\int_{\mathbb{R}^N} \bigg| \frac{\nabla^m u_0}{\|u_0\|_{L^2}}\bigg|^2 \, dx\right)^{\frac{N}{N-2m}} = C_{N, \log}  e^{(2- 2^{*} ) \alpha_2} \int_{\mathbb{R}^N} \bigg|\frac{u_0}{\|u_0\|_{L^2}}\bigg|^2 \log \left| \frac{e^{\alpha_2}}{\|u_0\|_{L^2} } u_0 \right| \, dx  \\ & 
   \frac{N}{N-2m}\log\bigg(\int_{\mathbb{R}^N} \bigg| \frac{\nabla^m u_0}{\|u_0\|_{L^2}} \bigg|^2 \, dx \bigg)= \log C_{N, \log} + \log e^{(2- 2^{*} ) \alpha_2}+ \log\bigg(\alpha_2 +\int_{\mathbb{R}^N} \bigg|\frac{u_0}{\|u_0\|_{L^2}}\bigg|^2 \log \left| \frac{u_0}{\|u_0\|_{L^2} } \right| \, dx  \bigg)
\end{aligned} \end{equation*}
sice $\alpha_2 = \frac{N-2m}{4m} -  \int_{\mathbb{R}^N} \bigg|\frac{u_0}{\|u_0\|_{L^2}}\bigg|^2 \log \left|\frac{u_0}{\|u_0\|_{L^2}}\right| \, dx$
\begin{equation*}\begin{aligned}
    &\frac{N}{N-2m}\log\bigg(\int_{\mathbb{R}^N}\bigg| \frac{\nabla^m u_0}{\|u_0\|_{L^2}}\bigg|^2 \bigg) \\& =\log C_{N, \log} + (2- 2^{**} ) \bigg(\frac{N-2m}{4m} -  \int_{\mathbb{R}^N} \bigg|\frac{u_0}{\|u_0\|_{L^2}}\bigg|^2 \log \left|\frac{u_0}{\|u_0\|_{L^2}}\right| \, dx\bigg) + \log\bigg(\frac{N-2m}{4m}  \bigg) \\ & \frac{N}{N-2m}\log\bigg(\int_{\mathbb{R}^N} \bigg|\frac{\nabla^m u_0}{\|u_0\|_{L^2}}\bigg|^2 \, dx \bigg)= \log \bigg( C_{N, \log } \frac{N-2m}{4m} \bigg) -1+ \frac{4m}{N-2m}\int_{\mathbb{R}^N} \bigg|\frac{u_0}{\|u_0\|_{L^2}}\bigg|^2 \log \left|\frac{u_0}{\|u_0\|_{L^2}}\right| \, dx \\ & \frac{N}{4m}\log\bigg(\int_{\mathbb{R}^N} \bigg| \frac{\nabla^m u_0}{\|u_0\|_{L^2}}\bigg|^2 \bigg)= \frac{N-2m}{4m}\log \bigg( C_{N, \log } \frac{N-2m}{4m}e^{-1} \bigg)+\int_{\mathbb{R}^N} \bigg|\frac{u_0}{\|u_0\|_{L^2}}\bigg|^2 \log \left|\frac{u_0}{\|u_0\|_{L^2}}\right| \, dx \\ & \log\bigg(\int_{\mathbb{R}^N} \bigg| \frac{\nabla^m u_0}{\|u_0\|_{L^2}}\bigg| ^2\, dx \bigg)^{\frac{N}{4m}}-\log \bigg( C_{N, \log } \frac{N-2m}{4m}e^{-1} \bigg)^{ \frac{N-2m}{4m}}= \int_{\mathbb{R}^N} \bigg|\frac{u_0}{\|u_0\|_{L^2}}\bigg|^2 \log \left|\frac{u_0}{\|u_0\|_{L^2}}\right| \, dx \\ & \log \bigg( \frac{\bigg(\int_{\mathbb{R}^N} \bigg| \frac{\nabla^m u_0}{\|u_0\|_{L^2}}\bigg|^2 \, dx  \bigg)^{\frac{N}{4m}}}{\bigg( C_{N, \log } \frac{N-2m}{4m}e^{-1} \bigg)^{ \frac{N-2m}{4m}}}\bigg)=\int_{\mathbb{R}^N} \bigg|\frac{u_0}{\|u_0\|_{L^2}}\bigg|^2 \log \left|\frac{u_0}{\|u_0\|_{L^2}}\right| \, dx\\ & \frac{N}{4m} \log\bigg( \bigg(  \frac{4m e}{C_{N,\log}(N-2m)}   \bigg)^{\frac{N-2m}{N}} \int_{\mathbb{R}^N} \bigg|\frac{\nabla^m u_0}{\|u_0\|_{L^2}}\bigg|^2\, dx \bigg)=\int_{\mathbb{R}^N} \bigg|\frac{u_0}{\|u_0\|_{L^2}}\bigg|^2 \log \left|\frac{u_0}{\|u_0\|_{L^2}}\right| \, dx  \end{aligned} \end{equation*}Therefore we obtain the equality for $\frac{u_0}{\|u_0\|_{L^2}}$ in \eqref{eq:1.7}. \\ 
    Let us suppose that \begin{equation*}
        \frac{N}{4m}\bigg(\bigg( \frac{4m e }{C_{N,\log}(N-2m)}\bigg)^{\frac{N-2m}{2m}} \int_{\mathbb{R}^N} |\nabla^m  u|^2 \, dx\bigg) = \int_{\mathbb{R}^N} |u|^2 \log |u|\, dx 
    \end{equation*} for same $u \in \mathcal{D}^{m,2}(\mathbb{R}^N)$ such that $\|u_0\|_{L^2}=1$. Then \begin{equation*}
        \left(\int_{\mathbb{R}^N} |\nabla^m u|^2 \, dx \right)^{\frac{N}{N-2m}} = C_{N, \log} \int_{\mathbb{R}^N} |u|^2 \log |u| \, dx
    \end{equation*} for $\alpha_1= \frac{N-2m}{4m}- \int_{\mathbb{R}^N}|u|^2\log |u|\, dx$ and the equality in \eqref{eq:6.1} holds for $u_1=e^{\alpha}u$. Hence $J(u_0)=\inf\limits_{\mathcal{M}} J$ for \begin{equation*}
        u_0:=u_1(r\cdot)\in \mathcal{M}, \, \, \textit{where} \, \, r=\bigg( 2^{*} \frac{\int_{\mathbb{R}^N} |u_1|^2 \log|u|\, dx}{\int_{\mathbb{R}^N}|\nabla^mu_1|^2\, dx}\, \bigg)^{\frac{1}{2m}}. 
    \end{equation*}
    Let us sketch the proof that $u_0$ is a critical point of $J$. Firstly, note that, for every $v \in C_0^\infty(\mathbb{R}^n)$, $G(u_0 + tv) \in L^1(\mathbb{R}^n)$, for $G(s) := s^2 \log |s|$. Fix an arbitrary $v \in C_0^\infty(\mathbb{R}^n)$. We use the fact that $G$ is $C^1$-smooth and the Lebesgue dominated convergence theorem to get
\[
\lim_{t \to 0} \frac{1}{t} \left( \int_{\mathbb{R}^n} G(u_0 + tv) \, dx - \int_{\mathbb{R}^n} G(u_0) \, dx \right) = \int_{\mathbb{R}^n} g(u_0)v \, dx.
\]

By the continuity, $\int_{\mathbb{R}^n} G(u_0 + tv) \, dx > 0$, for sufficiently small $|t| > 0$, so $(u_0 + tv)(x) \in \mathcal{M}$, where
\[
r = \left( \frac{2^* \int_{\mathbb{R}^n} G(u_0 + tv) \, dx}{\int_{\mathbb{R}^n} |\nabla^m (u_0 + tv)|^2 \, dx} \right)^{1/2m}.
\]

Hence
\[
J((u_0 + tv)(r)) \geq \inf_{\mathcal{M}} J = J(u_0),
\]
or, equivalently,
\[
\left( \frac{1}{2} - \frac{1}{2^*} \right) \int_{\mathbb{R}^n} |\nabla^m (u_0 + tv)|^2 \, dx \geq J(u_0)^{2m/N} \left( \int_{\mathbb{R}^n} G(u_0 + tv) \, dx \right)^{(N-2m)/N}.
\]
We then proceed similarly as in the last part of the proof of Lemma \ref{lem:theta} to conclude that
\[
\int_{\mathbb{R}^n} \nabla^m u_0 \nabla^mv \, dx = \int_{\mathbb{R}^n} g(u_0) v \, dx,
\]
which yields that $u_0$ is a critical point of $J$.

Finally, we show the estimate of the constant $C_{N, \log}$ from Theorem \ref{thm:2}. Observe that if $u \in \D$ and $\int_{\mathbb{R}^n} |u|^2 \, dx = 1$, in view of Lemma \ref{Poly-inequality} and the logarithmic Sobolev inequality \eqref{classical_log_ineq}, we obtain
\[
\int_{\mathbb{R}^n} |u|^2 \log(|u|) \, dx \leq \frac{N}{4} \log \left( \frac{2}{\pi e N} \left( \int_{\mathbb{R}^n} |\nabla^m u|^2 \, dx \right)^{1/2} \right)=
\frac{N}{4m} \log \left( \left( \frac{2}{\pi e N} \right)^m\int_{\mathbb{R}^n} |\nabla^m u|^2 \, dx \right).
\]

Thus, we find
\[
\left( \frac{4me}{C_{N, \log}(N-2m)} \right)^{(N-2m)/N} \leq \left( \frac{2}{\pi e N} \right)^m.
\]
\end{proof}

\section*{Acknowledgements}
The first two authors are supported by PRIN PNRR  P2022YFAJH {\sl \lq\lq Linear and Nonlinear PDEs: New directions and applications''} (CUP H53D23008950001), and partially supported by INdAM-GNAMPA.  
The second author thanks acknowledge financial
support from PNRR MUR project PE0000023 NQSTI - National Quantum Science and Technology
Institute (CUP H93C22000670006). The third author was partly supported by the National Science Centre, Poland (Grant No. 2023/51/B/ST1/00968).
\bigskip




\end{document}